\documentclass[a4paper,reqno,12pt]{amsart}

\numberwithin{equation}{subsection}

\usepackage{amssymb}
\usepackage{amstext}
\usepackage{amsmath}
\usepackage{amscd}
\usepackage{amsthm}
\usepackage{amsfonts}
\usepackage{enumerate}
\usepackage{graphicx}
\usepackage{latexsym}
\usepackage{mathrsfs}
\usepackage[all,color]{xy}
\xyoption{all}
\usepackage[usenames,dvipsnames]{xcolor}

\usepackage{pstricks}
\usepackage{lscape}
\usepackage{comment}

\newtheorem{theorem}{Theorem}[section]

\newtheorem{corollary}[theorem]{Corollary}
\newtheorem{lemma}[theorem]{Lemma}
\newtheorem{proposition}[theorem]{Proposition}
\newtheorem{definition-proposition}[theorem]{Definition-Proposition}

\newtheorem{conjecture}[theorem]{Conjecture}

\theoremstyle{definition}

\newtheorem{remark}[theorem]{Remark}

\newcommand{\ca}{{\mathcal A}}

\newcommand{\cc}{{\mathcal C}}

\newcommand{\cd}{\mathsf{D}}
\newcommand{\ce}{{\mathcal E}}
\newcommand{\cf}{{\mathcal F}}
\newcommand{\ch}{{\mathcal H}}
\newcommand{\ck}{\mathsf{K}}
\newcommand{\cm}{{\mathcal M}}
\newcommand{\cn}{{\mathcal N}}
\newcommand{\ct}{{\mathcal T}}
\newcommand{\cs}{{\mathcal S}}
\newcommand{\cp}{{\mathcal P}}
\newcommand{\cu}{{\mathcal U}}
\newcommand{\cx}{{\mathcal X}}
\newcommand{\cy}{{\mathcal Y}}
\newcommand{\cz}{{\mathcal Z}}

\newcommand{\MM}{{\mathcal M}}
\newcommand{\NN}{{\mathcal N}}

\renewcommand{\SS}{{\mathcal S}}

\newcommand{\Z}{\mathbb{Z}}

\newcommand{\bo}{\operatorname{b}\nolimits}

\newcommand{\fd}{\mathsf{fd}}
\newcommand{\Mod}{\mathsf{Mod}}

\renewcommand{\mod}{\mathsf{mod}}
\newcommand{\proj}{\mathsf{proj}}
\newcommand{\inj}{\mathsf{inj}}
\newcommand{\CM}{\mathsf{CM}}

\newcommand{\per}{\mathsf{per}}
\newcommand{\Ext}{\operatorname{Ext}\nolimits}
\newcommand{\Hom}{\operatorname{Hom}\nolimits}
\newcommand{\End}{\operatorname{End}\nolimits}

\newcommand{\op}{\operatorname{op}\nolimits}
\newcommand{\RHom}{\mathbf{R}\strut\kern-.2em\operatorname{Hom}\nolimits}

\newcommand{\twosilt}{\operatorname{2-silt}\nolimits}

\newcommand{\twoctilt}{2\strut\kern-.2em\operatorname{-ctilt}\nolimits}
\newcommand{\dctilt}{{\it d}\strut\kern-.2em\operatorname{-ctilt}\nolimits}
\newcommand{\silt}{\operatorname{silt}\nolimits}
\newcommand{\presilt}{\operatorname{presilt}\nolimits}
\newcommand{\thick}{\mathsf{thick}}

\newcommand{\add}{\mathsf{add}}

\begin{document}

\title[Silting reduction and CY reduction]{Silting reduction and Calabi--Yau reduction of triangulated categories}

\author{Osamu Iyama}
\address{O. Iyama: Graduate School of Mathematics, Nagoya University, Chikusa-ku, Nagoya, 464-8602 Japan}
\email{iyama@math.nagoya-u.ac.jp}

\author{Dong Yang}
\address{D. Yang: Department of Mathematics, Nanjing University, 22 Hankou Road, Nanjing 210093, P. R. China}
\email{yangdong@nju.edu.cn}
\date{\today}

\begin{abstract}
We study two kinds of reduction process of triangulated categories, that is, silting reduction and Calabi-Yau reduction.
It is shown that the silting reduction $\ct/\thick\cp$ of a triangulated category $\ct$ with respect to a presilting subcategory $\cp$ can be realized as a certain subfactor category of $\ct$, and that there is a one-to-one correspondence between the set of (pre)silting subcategories of $\ct$ containing $\cp$ and the set of (pre)silting subcategories of $\ct/\thick\cp$. This result is applied to show that Amiot--Guo--Keller's construction of $d$-Calabi--Yau triangulated categories with $d$-cluster-tilting objects takes silting reduction to Calabi--Yau reduction.
\end{abstract}
\maketitle

\tableofcontents

\section{Introduction}
Derived categories and triangulated categories are ubiquitous in mathematics,  appearing in various areas such as representation theory, algebraic geometry, algebraic topology and mathematical physics.
One of the standard tools to study these categories is tilting theory, which enables us to control equivalences of triangulated categories.
Recently cluster tilting theory, a certain analog of tilting theory in Calabi--Yau triangulated categories, played an important role in the  categorification of cluster algebras of Fomin and Zelevinsky.
Central notions in these theories are silting objects and cluster tilting objects, which admit a categorical operation called mutation to construct a new object from a given one by replacing a direct summand.
It is known that the class of silting objects parametrizes other important structures in a given triangulated category, including co-t-structures, t-structures and simple-minded collections \cite{KellerNicolas11,KY,BY}.

The aim of this paper is to develop further a certain aspect of tilting theory and cluster tilting theory by focusing on two kinds of reduction process of triangulated categories which were studied in representation theory.
One is called \emph{Calabi--Yau reduction}, introduced in  \cite{IyamaYoshino08} (see also \cite{IyamaWemyss}). This is defined for
a $d$-rigid subcategory $\cp$ of a $d$-Calabi--Yau triangulated category $\ct$
as a certain subfactor category $\cu$ of $\ct$.
In this case $\cu$ is again a $d$-Calabi--Yau triangulated category,
and there is a natural bijection between $d$-cluster-tilting subcategories of $\ct$
containing $\cp$ and $d$-cluster-tilting subcategories of $\cu$.

The other one is called \emph{silting reduction}. 
This is defined for a presilting subcategory $\cp$ of a triangulated category $\ct$ as the triangle quotient $\cu=\ct/\thick\cp$. 
Under certain mild conditions (P1) and (P2) in Section~\ref{ss:add-equiv}, our first main result enables us to realise $\cu$ inside of $\ct$ as a certain subfactor category, which is much easier to control than triangle quotients and analogous to Calabi--Yau reduction.

\begin{theorem}[{Theorems \ref{equivalence} and \ref{t:triangle-equivalence}}]\label{thm:main-thm-1}
Let $\ct$ be a triangulated category, $\cp$ a presilting subcategory of $\ct$ satisfying (P1) and (P2) and $\cu=\ct/\thick\cp$.
Then the additive quotient $\frac{\cz}{[\cp]}$ for $\cz=({}^{\perp_\ct}\cp[>\hspace{-3pt}0])\cap(\cp[<\hspace{-3pt}0]{}^{\perp_\ct})$ has a natural structure of a triangulated category (given in Theorem \ref{triangle structure of Z/P}), and we have a triangle equivalence
$\frac{\cz}{[\cp]}\stackrel{\simeq}{\longrightarrow}\cu$.
\end{theorem}

We recover, as a special case of this realisation, the well-known triangle equivalence due to Buchweitz \cite{Buchweitz87}
\[\underline{\CM} A\xrightarrow{\simeq}\cd^{\rm b}(\mod A)/\ck^{\rm b}(\proj A)\]
for an Iwanaga--Gorenstein ring $A$ (Example \ref{buchweitz}). Moreover, there is a natural bijection between silting subcategories of $\ct$
containing $\cp$ and silting subcategories of $\cu$ (Theorem \ref{t:silting-reduction}), which preserves a canonical partial order on the set of silting subcategories (Corollary \ref{iso of posets}).
A similar result was given in \cite[Theorem 2.37]{AI} under the strong restriction that $\thick\cp$ is functorially finite in $\ct$. We can drop this assumption thanks to the realisation of $\cu$ as a subfactor category of $\ct$.

The second main result of this paper is to compare these two reduction processes using Amiot--Guo's construction \cite{Amiot09,Guolingyan11a} (based on Keller's work \cite{Keller05,Keller11}), which is a direct passage from tilting theory to cluster tilting theory. 
Let $\ct$ be a triangulated category, $\cm$ a subcategory of $\ct$ and $\ct^{\fd}\subset\ct$ a triangulated subcategory such that $(\ct,\ct^{\fd},\cm)$ is a $(d+1)$-\emph{Calabi--Yau triple} (see Section~\ref{ss:setup} for the precise definition). We fix a functorially finite subcategory $\cp$ of $\cm$. On the one hand, applying Amiot--Guo--Keller's construction, we obtain a $d$-Calabi--Yau triangulated category $\cc=\ct/\ct^{\fd}$ in which $\cp$ becomes a $d$-rigid subcategory. Then we form the Calabi--Yau reduction $\cc_\cp$ of $\cc$ with respect to $\cp$, which is $d$-Calabi--Yau and in which $\cm$ becomes a $d$-cluster-tilting subcategory. 
On the other hand, we first form the silting reduction $\cu=\ct/\thick\cp$, which turns out to be part of a relative $(d+1)$-Calabi--Yau triple $(\cu,\cu^{\fd},\cm)$. Then Amiot--Guo--Keller's construction yields a $d$-Calabi--Yau triangulated category $\cu/\cu^{\fd}$ in which $\cm$ becomes a $d$-cluster-tilting subcategory. 
We prove that the two resulting $d$-Calabi--Yau triangulated categories $\cc_\cp$ and $\cu/\cu^{\fd}$ are triangle equivalent (Theorem~\ref{main theorem new}). In this sense, Amiot--Guo--Keller's construction takes silting reduction to Calabi--Yau reduction. This can  be illustrated by the following commutative diagram of operations.
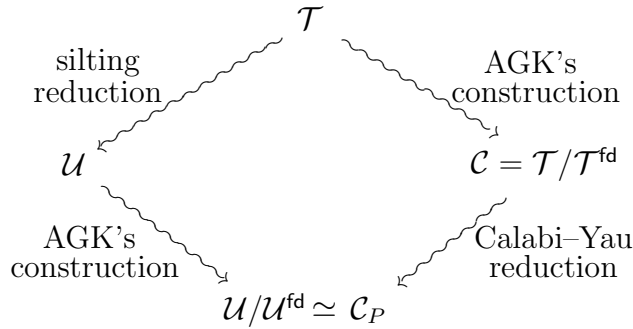
\[
 \SelectTips{cm}{10}
\begin{xy}
 0;<1.1pt,0pt>:<0pt,-1.1pt>::
(80,0) *+{\ct} ="0", (0,50) *+{\cu} ="1", (160,50) *+{\cc=\ct/\ct^{\fd}} ="2",
(65,100) *+{\cu/\cu^{\fd}} ="3", (85,100) *+{\simeq} ="4", (100,100) *+{\cc_P} ="5",
(155,15) *+{\text{AGK's}}="6",
(157,25) *+{\text{construction}}="7",
(8,15) *+{\text{silting}}="8",
(8,25) *+{\text{reduction}}="9",
(5,75) *+{\text{AGK's}}="10",
(7,85) *+{\text{construction}}="11",
(163,75) *+{\text{Calabi--Yau}}="12",
(163,85) *+{\text{reduction}}="13",
(220,0) *+{},
"0", {\ar@{~>}, "1"}, {\ar@{~>}, "2"},
"1", {\ar@{~>}, "3"}, 
"2", {\ar@{~>}, "5"},
\end{xy}
\]
The case when $\ct$ is the perfect derived category of a Ginzburg differential graded (=dg) algebra was studied by Keller in \cite[Section 7]{Keller11}.  The diagram above induces a commutative diagram of maps
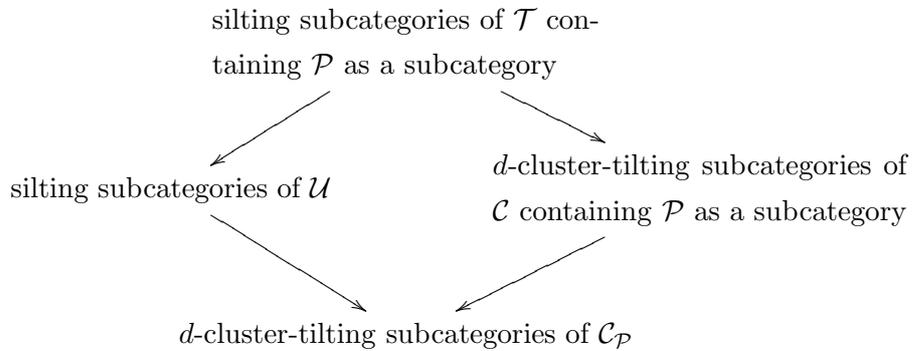
\[
\begin{xy}
 0;<1.1pt,0pt>:<0pt,-1.1pt>::
(100,0) *+{{\parbox{145pt}{\small silting subcategories of $\ct$ containing $\cp$ as a subcategory}}} ="0", (20,50) *+{{\parbox{120pt}{\small silting subcategories of $\cu$}}} ="1", 
(200,50) *+{{\parbox{155pt}{\small $d$-cluster-tilting subcategories of $\cc$ containing $\cp$ as a subcategory}}} ="2",
(100,100) *+{{\parbox{170pt}{\small $d$-cluster-tilting subcategories of $\cc_\cp$}}} ="3", 
"0", {\ar "1"}, {\ar "2"},
"1", {\ar "3"}, 
"2", {\ar "3"},
\end{xy}
\]
where the two left-going maps are bijections due to respective properties of silting reduction and Calabi--Yau reduction. 

Moreover if $\cm$ has an additive generator, then the two right-going maps above are surjections for $d=1$ and for $d=2$ (due to Keller--Nicol\'as \cite{KellerNicolas11} in the algebraic setting)  (Corollary~\ref{2term-silting and cluster-tilting}).

To prove our results in Section~\ref{s:silting-red-vs-cy-red}, we will prepare in Section~\ref{section:adjacent t-structures} some general observations on t-structures in triangulated categories, which has its own importance.
It is known that any silting subcategory $\cm$ in a triangulated category $\ct$ gives rise to a co-t-structure $(\ct_{\ge0},\ct_{\le0})$ in $\ct$ (see Proposition~\ref{from silting to co-t-structure} for details).
We study the condition that there is a t-structure $(\cx,\cy)$ in $\ct$ satisfying $\cx=\ct_{\le0}$. We prove that this condition is invariant under a suitable change of the silting subcategory $\cm$ (Theorem~\ref{l:shift-t-str 1}).
Moreover, under certain conditions, we prove that this condition is equivalent to its dual, that is, there is a t-structure $(\cx',\cy')$ in $\ct$ satisfying $\cy'=\ct_{\ge0}$ (Theorem~\ref{left-right 1}).
This result is used to simplify the proofs of Amiot--Guo--Keller's fundamental results (Theorem \ref{amiot}).

We remark that more general versions of Theorem~\ref{thm:main-thm-1} have since been established in \cite{Wei15,Li16,N,IY}.  We refer to the work \cite{Jasso} of Jasso for a reduction of support $\tau$-tilting modules and its connection with our silting reduction.

\smallskip
\noindent{\it Acknowledgement:}
The first-named author acknowledges financial support from JSPS Grant-in-Aid for Scientific Research (B) 24340004, (C) 23540045 and (S) 22224001.
The second-named author acknowledges financial support from a JSPS postdoctoral fellowship program (P12318) and from National Natural Science Foundation of China No. 11371196 and No. 11301272. 
Both authors thank Martin Kalck,  Huanhuan Li, Jorge Vit\'oria and Wuzhong Yang for helpful comments and inspiring discussions. They are grateful to Xiao-Wu Chen for pointing out the application Corollary~\ref{cor:relative-singularity-cat} and the fact that Theorems~\ref{equivalence} and \ref{t:triangle-equivalence} generalise Buchweitz' result. They thank a referee for helpful comments which make the paper more readable.

\section{Preliminaries}\label{s:preliminary}

In this section, we fix some notation. We recall the triangle structure of an additive quotient associated to a mutation pair.  We recall the definitions of silting subcategories, silting reduction, cluster-tilting subcategories, Calabi--Yau reduction, t-structures and co-t-structures. We recall derived categories of differential graded (=dg) algebras and Keller's Morita theorem for triangulated categories.

\subsection{Some notation}\label{ss:notation}
For a ring $R$, we denote by $\mod R$ the category of finitely generated right $R$-modules, by $\proj R$ the category of finitely generated projective right $R$-modules, by $\cd^{\rm b}(\mod R)$ the bounded derived category of $\mod R$ and by $\ck^{\rm b}(\proj R)$ the bounded homotopy category of $\proj R$.

Let $\ct$ be an additive category. For morphisms $f:X\to Y$ and $g:Y\to Z$, we denote by $gf:X\to Z$ the composition. We say that $\ct$ is \emph{idempotent complete} if any idempotent morphism $e:X\rightarrow X$ has a kernel. Let $\cs$ be a full subcategory of $\ct$ (for example, an object of $\ct$ will often be considered as a full subcategory with one object). For an object $X$ of $\ct$, we say that a morphism $f:S\to X$ is a \emph{right $\cs$-approximation} of $X$ if $S\in\cs$ and $\Hom_\ct(S',f)$ is surjective for any $S'\in\cs$. We say that $\cs$ is \emph{contravariantly finite} if every object in $\ct$ has a right $\cs$-approximation. Dually, we define \emph{left $\cs$-approximations} and \emph{covariantly finite} subcategories. We say that $\cs$ is \emph{functorially finite} if it is both contravariantly finite and covariantly finite \cite{AuslanderSmalo81}.
For example, if $\ct$ satisfies the following finiteness condition (F), then $\add X$ is a functorially finite subcategory of $\ct$ for any $X\in\ct$.
\begin{itemize}
\item[(F)] $\Hom_\ct(X,Y)$ is finitely generated as an $\End_\ct(X)$-module and as an $\End_\ct(Y)^{\op}$-module.
\end{itemize}
This condition (F) is satisfied if $\ct$ is $k$-linear and Hom-finite for a commutative ring $k$.

Denote by $\add_\ct\cs$ (or simply $\add\cs$) the smallest full subcategory of $\ct$ which contains $\cs$ and which is closed under taking isomorphisms, finite direct sums and direct summands.
Denote by $[\cs]$ the ideal of $\ct$ consisting of morphisms which factor through an object of $\add_{\ct}\cs$  and denote by $\frac{\ct}{[\cs]}$ the corresponding additive quotient of $\ct$ by $\cs$.
Define full subcategories
\begin{eqnarray*}
{}^{\perp_\ct}\cs&:=&\{X\in\ct\ |\ \Hom_{\ct}(X,\cs)=0\},\\
\cs^{\perp_\ct}&:=&\{X\in\ct\ |\ \Hom_{\ct}(\cs,X)=0\}.
\end{eqnarray*}
When it does not cause confusion, we will simply write ${}^\perp\cs$ and $\cs^\perp$.

\smallskip

Let $\ct$ be a triangulated category. We will denote by $[1]$ the shift functor of any triangulated category unless otherwise stated. For two objects $X$ and $Y$ of $\ct$ and an integer $n$, by $\Hom_\ct(X,Y[>\hspace{-3pt}n])=0$ (respectively, $\Hom_\ct(X,Y[\ge\hspace{-3pt}n])=0$, $\Hom_\ct(X,Y[<\hspace{-3pt}n])=0$, $\Hom_\ct(X,Y[\le\hspace{-3pt}n])=0$), we mean $\Hom_\ct(X,Y[i])=0$ for all $i>n$ (respectively, for all $i\geq n$, $i<n$, $i\leq n$).

Let $\cs$ be a full subcategory of $\ct$. We say that $\cs$ is a \emph{thick subcategory} of $\ct$ if it is a triangulated subcategory of $\ct$ which is closed under taking direct summands. In this case, we denote by $\ct/\cs$ the triangle quotient of $\ct$ by $\cs$. In general, we denote by $\thick_\ct\cs$ (or simply $\thick\cs$) the smallest thick subcategory of $\ct$ which contains $\cs$.

Let $\cs$ and $\cs'$ be full subcategories of $\ct$.  By $\Hom_\ct(\cs,\cs')=0$, we mean $\Hom_\ct(S,S')=0$ for all $S\in\cs$ and $S'\in\cs'$. Define
\begin{align*}
\cs*\cs'=\cs*_\ct\cs':=&\{X\in\ct\ |\ \text{there is a triangle } S\rightarrow X\rightarrow S'\rightarrow S[1]\\
& \text{ with } S\in\cs \text{ and } S'\in\cs'\}.
\end{align*}

\subsection{Mutation pairs and cluster-tilting subcategories}\label{section: Mutation pair}
Let $\ct$ be a triangulated category. Let $\cp$ be a full subcategory of $\ct$ such that $\Hom_\ct(\cp,\cp[1])=0$ and let $\cz$ be an extension-closed full subcategory of $\ct$ which contains $\cp$.
Assume that $(\cz,\cz)$ forms a \emph{$\cp$-mutation pair} in the sense of  \cite{IyamaYoshino08}, i.e. the following conditions are satisfied:
\begin{itemize}
\item $\cp\subset\cz$ and $\Hom_\ct(\cp,\cz[1])=0=\Hom_\ct(\cz,\cp[1])$.
\item For any $Z\in\cz$, there exists triangles $Z\to P'\to Z'\to Z[1]$ and $Z''\to P''\to Z\to Z''[1]$ with $P',P''\in\cp$ and $Z',Z''\in\cz$.
\end{itemize}

\begin{theorem}\label{triangle structure of Z/P}
\emph{(\cite[Theorem 4.2]{IyamaYoshino08})}
The category $\frac{\cz}{[\cp]}$ has the structure of a triangulated category
with respect to the following shift functor and triangles:
\begin{itemize}
\item[(a)] For $X\in\cz$, we take a triangle
\[\xymatrix{X\ar[r]^{\iota_X}&P_X\ar[r]& X\langle1\rangle\ar[r]& X[1]}\]
with a (fixed) left $\cp$-approximation $\iota_X$.
Then $\langle1\rangle$ gives a well-defined auto-equivalence of $\frac{\cz}{[\cp]}$,
which is the \emph{shift functor} of $\frac{\cz}{[\cp]}$.
\item[(b)] For a triangle $X\xrightarrow{f} Y\xrightarrow{g} Z\xrightarrow{h}X[1]$
with $X,Y,Z\in\cz$, take the following commutative diagram of triangles:
\begin{equation}\label{define a triangle}
\xymatrix{
X\ar[r]^f\ar@{=}[d]&Y\ar[r]^g\ar[d]&Z\ar[r]^h\ar[d]^a&X[1]\ar@{=}[d]\\
X\ar[r]^{\iota_X}&P_X\ar[r]&X\langle1\rangle\ar[r]&X[1]
}\end{equation}
Then we have a complex $X\xrightarrow{\overline{f}}Y\xrightarrow{\overline{g}}Z
\xrightarrow{\overline{a}}X\langle1\rangle$. We define \emph{triangles} in
$\frac{\cz}{[\cp]}$ as the complexes which are isomorphic to complexes obtained in this way.
\end{itemize}
\end{theorem}

Let $k$ be a field and $\ct$ be a $k$-linear triangulated category. Let $d\geq 1$ be an integer. Then
$\ct$ is said to be \emph{$d$-Calabi--Yau} if $\ct$ is Hom-finite, and there is a bifunctorial isomorphism for any objects $X$ and $Y$ of $\ct$:
\[
D\Hom_\ct(X,Y)\simeq \Hom_\ct(Y,X[d]),
\]
where $D=\Hom_k(-,k)$ is the $k$-dual.

Assume that $\ct$ is $d$-Calabi--Yau. A full subcategory $\cp$ of $\ct$ is \emph{$d$-rigid} if $\Hom_{\ct}(\cp,\cp[i])=0$ for all $1\leq i\leq d-1$.
It is \emph{$d$-cluster-tilting}  if $\cp$ is functorially finite and 
the following equivalence holds for $X\in\ct$:
\[
\Hom_\ct(\cp,X[i])=0 \text{ for all } 1\leq i\leq d-1\Longleftrightarrow X\in\add\cp.
\]
By \cite[Theorem 3.1(1)]{IyamaYoshino08}, a $d$-rigid subcategory $\cp$ of $\ct$ is $d$-cluster-tilting if and only if $\ct=\cp*\cp[1]*\cdots*\cp[d-1]$ holds.
An object $P$ of $\ct$ is \emph{$d$-rigid} if $\add P$ is a $d$-rigid subcategory, and \emph{$d$-cluster-tilting} if $\add P$ is a $d$-cluster-tilting subcategory.
We point out that $\add P$ is always functorially finite.

Let $\cp$ be a functorially finite $d$-rigid subcategory of $\ct$. Let
\[
\cz:={}^{\perp_\ct} (\cp[1]*\cp[2]*\cdots*\cp[d-1])\ \mbox{ and }\ \ct_\cp:=\frac{\cz}{[\add\cp]}.
\]
Then the additive category $\ct_\cp$, called the \emph{Calabi--Yau reduction} of $\ct$ with respect to $\cp$ in \cite{IyamaYoshino08}, carries a natural structure of a triangulated category, by Theorem~\ref{triangle structure of Z/P}. Moreover,

\begin{theorem} \emph{(\cite[Theorem 4.9]{IyamaYoshino08})}\label{t:cy-reduction}
The projection functor $\cz\rightarrow\ct_\cp$ induces a one-to-one correspondence between the set of $d$-cluster-tilting subcategories of $\ct$ which contains $\cp$ and the set of $d$-cluster-tilting subcategories of $\ct_\cp$.
\end{theorem}

We will use the following cluster-Beilinson criterion for triangle equivalence due to Keller--Reiten. 

\begin{proposition} \emph{(\cite[Lemma 4.5]{KellerReiten08})}\label{cluster-Beilinson}
Let $\ct'$ be another $d$-Calabi--Yau triangulated category and let $\cp\subset\ct$ and $\cp'\subset\ct'$ be $d$-cluster-tilting subcategories and $F:\ct\to\ct'$ be a triangle functor.
If $F$ induces an equivalence $\cp\to\cp'$, then $F$ is a triangle equivalence.
\end{proposition}

\subsection{Presilting and silting subcategories, t-structures and co-t-structures}\label{ss:silting-and-t-structure}
Let $\ct$ be a triangulated category. 

A full subcategory $\cp$ of $\ct$ is \emph{presilting} if $\Hom_{\ct}(\cp,\cp[i])=0$ for any $i>0$. It is \emph{silting} if in addition $\ct=\thick\cp$. An object $P$ of $\ct$ is \emph{presilting} if $\add P$ is a presilting subcategory and \emph{silting} if $\add P$ is a silting subcategory.

We denote by $\silt\ct$ (respectively, $\presilt\ct$) the class of silting (respectively, presilting) subcategories of $\ct$.
As usual we identify two (pre)silting subcategories $\MM$ and $\NN$ of $\ct$ when $\add\MM=\add\NN$.
The class $\silt\ct$ has a natural partial order: For $\MM,\NN\in\silt\ct$, we write
\[\MM\ge\NN\]
if $\Hom_{\ct}(\MM,\NN[>\hspace{-3pt}0])=0$.
This gives a partial order $\ge$ on $\silt\ct$, see \cite[Theorem 2.11]{AI}.

Triangulated categories with silting subcategories satisfy the following property.

\begin{lemma}\emph{(\cite[Proposition 2.4]{AI})}\label{vanishing for >>}
Let $\ct$ be a triangulated category with a silting subcategory $\cm$.
\begin{itemize}
\item[(a)] For any $X,Y\in\ct$, there exists $i\in\mathbb{Z}$ such that $\Hom_{\ct}(X,Y[\geq \hspace{-3pt}i])=0$.
\item[(b)] For any $X\in\ct$, there exist $i,j\in\mathbb{Z}$ such that $\Hom_{\ct}(\cm,X[\geq\hspace{-3pt}i])=0$ and  $\Hom_{\ct}(X,\cm[\geq\hspace{-3pt}j])=0$.
\end{itemize}
\end{lemma}


A \emph{torsion pair} of $\ct$ is a pair $(\cx,\cy)$ of full subcategories of $\ct$ such that
\begin{itemize}
\item[(T1)] $\cx={}^\perp\cy$ and $\cy=\cx^\perp$;
\item[(T2)] $\ct=\cx*\cy$, namely, for each $M\in\ct$ there is a triangle $X_M\rightarrow
M\rightarrow Y_M\rightarrow X[1]$ in $\ct$ with $X_M\in\cx$ and $Y_M\in\cy$.
\end{itemize}
It is elementary that the condition (T1) can be replaced by the following condition:
\begin{itemize}
\item[(T1$'$)] $\Hom_{\ct}(\cx,\cy)=0$, $\cx=\add\cx$ and $\cy=\add\cy$.
\end{itemize}

A \emph{t-structure} on
$\ct$ (\cite{BeilinsonBernsteinDeligne82}) is a pair $(\ct^{\leq 0},\ct^{\geq 0})$ of full subcategories of $\ct$ such that $\ct^{\geq 1}\subset \ct^{\geq 0}$ and $(\ct^{\leq 0},\ct^{\geq 1})$ is a torsion pair. Here for an integer $n$ we denote $\ct^{\leq n}=\ct^{\leq 0}[-n]$ and $\ct^{\geq n}=\ct^{\geq 0}[-n]$. In this case, the triangle in the second condition above is unique up to a unique isomorphism, and the assignments $M\mapsto X_M$ and $M\mapsto Y_M$ define two functors $\sigma^{\leq 0}:\ct\rightarrow\ct^{\leq 0}$ and $\sigma^{\geq 1}:\ct\rightarrow\ct^{\geq 1}$, called the \emph{truncation functors}. For an integer $n$ the pair $(\ct^{\leq n},\ct^{\geq n})$ is also a t-structure and we denote by $\sigma^{\leq n}$ and $\sigma^{\geq n+1}$ the associated truncation functors. 
The \emph{heart} $\ch:=\ct^{\leq 0}\cap\ct^{\geq 0}$ is
always an abelian category.
The t-structure $(\ct^{\leq 0},\ct^{\geq 0})$ is said to be
\emph{bounded} if
\[
\bigcup_{n\in\mathbb{Z}} \ct^{\leq n}=\ct=\bigcup_{n\in\mathbb{Z}}\ct^{\geq n},
\] 
equivalently, if $\ct=\thick\ch$.

A \emph{co-t-structure} on
$\ct$  (\cite{Pauksztello08,Bondarko10}) is a pair $(\ct_{\geq 0},\ct_{\leq 0})$ of full subcategories of $\ct$ such that $\ct_{\geq 1}\subset \ct_{\geq 0}$ and $(\ct_{\geq 1},\ct_{\leq 0})$ is a torsion pair. Here for an integer $n$ we denote $\ct_{\geq n}=\ct_{\geq 0}[-n]$ and $\ct_{\leq n}=\ct_{\leq 0}[-n]$. The \emph{co-heart} $\cp:=\ct_{\geq 0}\cap\ct_{\leq 0}$ is a presilting subcategory of $\ct$, but it is usually not an abelian category.
The co-t-structure $(\ct_{\geq 0},\ct_{\leq 0})$ is said to be \emph{bounded} if
\[
\bigcup_{n\in\mathbb{Z}} \ct_{\geq n}=
\ct=\bigcup_{n\in\mathbb{Z}}\ct_{\leq n},
\]
equivalently, if $\ct=\thick\cp$. The co-heart of a bounded co-t-structure is a silting subcategory of $\ct$.

\subsection{Results on additive closures, co-t-structures and idempotent completeness}

Throughout this subsection, let $\ct$ be an arbitrary triangulated category.  We give useful criterions for $\ct$ to be idempotent complete, and also for subcategories of $\ct$ to be closed under direct summands.

We start with preparing some easy observations, which will be used later.

\begin{lemma}\label{belong to add}
If $X\in\add(\cs*\cs')$ satisfies $\Hom_{\ct}(\cs,X)=0$, then $X\in\add\cs'$.
\end{lemma}

\begin{proof}
There exist $Y\in\ct$ and a triangle
\begin{equation}\label{triangle for a}
\xymatrix{S\ar[r]^(0.4){a}&X\oplus Y\ar[r]&S'\ar[r]&S[1]}
\end{equation}
with $S\in\add\cs$ and $S'\in\add\cs'$.
Since $\Hom_{\ct}(\cs,X)=0$, we can write $a={0\choose b}$ for $b:S\to Y$.
We extend $b$ to a triangle $S\xrightarrow{b}Y\stackrel{c}{\to} Z\to S[1]$. Then we have a triangle
\[\xymatrix@C=3em{S\ar[r]^(0.4){a={0\choose b}}&X\oplus Y\ar[r]^{{1_X\ 0\choose 0\ \ c}}&X\oplus Z\ar[r]& S[1].}\]
Comparing this with \eqref{triangle for a}, we have $S'\simeq X\oplus Z$. Thus $X\in\add\cs'$.
\end{proof}

Note that, if $\cs=\add\cs$ and $\cs'=\add\cs'$ hold, then $\cs*\cs'$ is closed under direct sums, but not necessarily under direct summands.
We have the following sufficient condition for the equality $\cs*\cs'=\add(\cs*\cs')$ to hold (cf. \cite[Proposition 2.1]{IyamaYoshino08} for the Krull--Schmidt case).

\begin{lemma}\label{closed under summands}
Let $\cs=\add\cs$ and $\cs'=\add\cs'$ be subcategories of $\ct$ satisfying $\Hom_{\ct}(\cs,\cs')=0$ and $\cs[1]\subset\cs'$. 
\begin{itemize}
\item[(a)] We have $\cs*\cs'=\add(\cs*\cs')$.
\item[(b)] If $\cs$ and $\cs'$ are idempotent complete, so is $\cs*\cs'$. 
\end{itemize}
\end{lemma}

\begin{proof}
Since $\cs$ and $\cs'$ are closed under direct sums, it follows easily from definition that $\cs*\cs'$ is also closed under direct sums.
It remains to show that $\cs*\cs'$ is closed under direct summands.
Assume that $X\oplus X'\in\cs*\cs'$, that is, there exists a triangle
\begin{equation}\label{middle X oplus X'}
\xymatrix{S\ar[r]^(.4){{a\choose a'}}&X\oplus X'\ar[r]^(.6){(b\ b')}&S'\ar[r]& S[1]}
\end{equation}
with $S\in\cs$ and $S'\in\cs'$. Now we extend $a:S\to X$ to a triangle
\begin{equation}\label{middle X}
\xymatrix{S\ar[r]^{a}&X\ar[r]^{c}& Y\ar[r]& S[1].}
\end{equation}
Since $\Hom_{\ct}(\cs,S')=0$, the map $\Hom_{\ct}(\cs,S)\xrightarrow{{a\choose a'}\cdot}\Hom_{\ct}(\cs,X\oplus X')$ is surjective by the triangle \eqref{middle X oplus X'}.
In particular, the map $\Hom_{\ct}(\cs,S)\xrightarrow{a\cdot}\Hom_{\ct}(\cs,X)$ is also surjective.
Thus we have $\Hom_{\ct}(\cs,Y)=0$ by the triangle \eqref{middle X} and our assumptions $\Hom_{\ct}(\cs,\cs')=0$ and $\cs[1]\subset\cs'$.

Using the octahedron axiom, we have the following commutative diagram:
\[\xymatrix{
&S\ar@{=}[r]\ar[d]_{{a\choose a'}}&S\ar[d]&\\
S\ar[r]^(.4){{a\choose 0}}\ar@{=}[d]&X\oplus X'\ar[r]^{{c\ \ 0\ \choose 0\ 1_{X'}}}\ar[d]_{(b\ b')}&Y\oplus X'\ar[r]\ar[d]&S[1]\ar@{=}[d]\\
S\ar[r]&S'\ar[r]\ar[d]&Z\ar[r]\ar[d]&S[1]\\
&S[1]\ar@{=}[r]&S[1]
}\]
Since $\Hom_{\ct}(S,S')=0$, the lower horizontal triangle splits, and we have $Z\simeq S'\oplus S[1]\in\cs'$.
Thus the right vertical triangle shows $Y\in\add(\cs*\cs')$.
Since $\Hom_{\ct}(\cs,Y)=0$ holds, we have $Y\in\add\cs'=\cs'$ by Lemma \ref{belong to add}. Therefore $X\in \cs*\cs'$. 

(b) Let $\ct^\omega$ be the idempotent completion of $\ct$. Then $\ct^\omega$ has a natural triangle structure such that $\ct$ becomes a triangulated subcategory of $\ct^\omega$ by \cite{BalmerSchlichting01}. 
Then $\cs*_\ct\cs'=\cs*_{\ct^\omega}\cs'$ since $\Hom_{\ct}(\cs',\cs[1])=\Hom_{\ct^\omega}(\cs',\cs[1])$. Since $\cs$ and $\cs'$ are idempotent complete, we have $\cs=\add_{\ct^\omega}\cs$ and $\cs'=\add_{\ct^\omega}\cs'$. So by Lemma~\ref{closed under summands}(a) 
\[
\cs*_\ct\cs'=\cs*_{\ct^\omega}\cs'=\add_{\ct^\omega}(\cs*_{\ct^\omega}\cs')
\]
is idempotent complete.
\end{proof}

We often use the following observation in this paper.

\begin{proposition}\label{additively closed}
Let $\ct$ be a triangulated category and $\cp=\add\cp$ a full subcategory of $\ct$ and $n\ge0$. Assume that
$\Hom_\ct(\cp,\cp[i])=0$ for any $i$ with $1\le i\le n$.
\begin{itemize}
\item[(a)] We have 
$\cp*\cp[1]*\cdots*\cp[n]=\add(\cp*\cp[1]*\cdots*\cp[n])$.
\item[(b)] If $\cp$ is idempotent complete, so is $\cp*\cp[1]*\cdots*\cp[n]$.
\end{itemize}
\end{proposition}

\begin{proof}
(a) For $n=0$ the assertion is the assumption $\cp=\add\cp$. Assume that it holds for $n-1$.
Then $\cs:=\cp$ and $\cs':=\cp[1]*\cp[2]*\cdots*\cp[n]$ satisfies
$\add\cs=\cs$ and $\add\cs'=\cs'$. In particular, the assumptions in Lemma \ref{closed under summands}(a) are satisfied, and hence
$\cs*\cs'=\cp*\cp[1]*\cdots*\cp[n]$ satisfies $\cs*\cs'=\add(\cs*\cs')$.

(b) Similarly this follows by induction on $n$ by using Lemma~\ref{closed under summands}(b).
\end{proof}

Now we show that any silting subcategory gives a co-t-structure on $\ct$. 
The following proposition is well-known, and was proved as \cite[Theorem 5.5]{MendozaSaenzSantiagoSouto13}, see also \cite[Proposition 2.22]{AI}, \cite[proof of Theorem 4.3.2]{Bondarko10} and \cite{KellerNicolas11}.

\begin{proposition}\label{from silting to co-t-structure}
Let $\ct$ be a triangulated category and $\cm$ a silting subcategory of $\ct$ with $\cm=\add\cm$.
\begin{itemize}
\item[(a)] Then $(\ct_{\ge0},\ct_{\le0})$ is a bounded co-t-structure on $\ct$, where
\begin{eqnarray*}
\ct_{\ge0}:=\bigcup_{n\ge0}\cm[-n]*\cdots*\cm[-1]*\cm\ \mbox{ and }\ 
\ct_{\le0}:=\bigcup_{n\ge0}\cm*\cm[1]*\cdots*\cm[n].
\end{eqnarray*}
\item[(b)] For any integers $m$ and $n$, we have
\[\ct_{\ge n}\cap\ct_{\le m}=\left\{\begin{array}{cc}
\cm[-m]*\cm[1-m]*\cdots*\cm[-n]&\mbox{ if }n\le m,\\
0&\mbox{ if }n>m.
\end{array}\right.\]
\end{itemize}
\end{proposition}

\begin{proof}
(a) For the convenience of the reader, we give a simple direct proof.
By induction we obtain $\Hom_{\ct}(\ct_{\ge1},\ct_{\le0})=0$.
Since $\Hom_{\ct}(\cm,\cm[>\hspace{-3pt}0])=0$, we have $\ct_{\ge1}=\add\ct_{\ge1}$ and $\ct_{\le0}=\add\ct_{\le0}$ by Proposition \ref{additively closed}. Thus the condition (T1$'$) holds.
On the other hand, there is the following equality
\[\ct=\bigcup_{n\ge0}\add(\cm[-n]*\cm[1-n]*\cdots*\cm[n-1]*\cm[n])\]
by \cite[Lemma 2.15(b)]{AI}. Applying Proposition \ref{additively closed} again, we have the condition (T2):
\[\ct=\bigcup_{n\ge0}\cm[-n]*\cm[1-n]*\cdots*\cm[n-1]*\cm[n]=
\ct_{\ge0}*\ct_{<0}.\]

(b) This can be shown easily by using Lemma \ref{belong to add}.
\end{proof}

As a consequence of Propositions~\ref{from silting to co-t-structure} and~\ref{additively closed}, we have

\begin{theorem}\label{c:idempotent-completeness-of-co-heart-implies-idempotent-completeness}
If a triangulated category has an idempotent complete silting subcategory (respectively, $d$-cluster-tilting subcategory for some $d\ge1$), then it is idempotent complete.
\end{theorem}

As a special case of Theorem~\ref{c:idempotent-completeness-of-co-heart-implies-idempotent-completeness}, we recover the well-known result that the bounded homotopy category of finitely generated projective modules over a ring is idempotent complete.
The `silting' part of Theorem~\ref{c:idempotent-completeness-of-co-heart-implies-idempotent-completeness} is \cite[Lemma 5.2.1]{Bondarko10}. It can be reformulated as: If $\ct$ has a bounded co-t-structure with idempotent complete co-heart, then $\ct$ is idempotent complete. It can be considered as `dual' to the fact that if $\ct$ has a bounded t-structure, then $\ct$ is idempotent complete (see \cite[Theorem]{ChenLe07}).

\subsection{Derived categories of dg algebras}\label{ss:derived-category}
We follow \cite{Keller94,Keller06d}.

Let $k$ be a field and $A$ be a dg ($k$-)algebra, that is, a graded algebra endowed with a compatible structure of a complex. A (right) dg $A$-module is a (right) graded $A$-module endowed with a compatible structure of a complex.
Let $\cd(A)$ denote the derived category of dg $A$-modules. This is a triangulated category whose shift functor is the shift of complexes.
Let $\per(A)=\thick(A_A)$ and let $\cd_{\rm fd}(A)$ denote the full subcategory of $\cd(A)$ consisting of dg $A$-modules whose total cohomology is finite-dimensional over $k$. These are two triangulated subcategories of $\cd(A)$.

Let $\ct$ be an algebraic triangulated category (over $k$), that is, $\ct$ is triangle equivalent to the stable category of a Frobenius category. Assume that $\ct$ is idempotent complete and $M$ is an object of $\ct$ such that $\ct=\thick(M)$.
Then by \cite[Theorem 3.8 b)]{Keller06d}, there is a dg algebra $A$ together with a triangle equivalence $\ct\to\per(A)$ which takes $M$ to $A_A$.
We briefly describe the construction of $A$ and refer to the proof of \cite[Theorem 4.3]{Keller94} for more details. Let $\ce$ be a Frobenius category such that the stable category of $\ce$ is triangle equivalent to $\ct$.
Let $\proj\ce$ denote the full subcategory of projective objects of $\ce$.
Then $\ck_{\rm ac}(\proj\ce)$, the homotopy category of acyclic complexes on $\proj\ce$, is triangle equivalent to $\ct$.
Let $\widetilde{M}$ be a preimage of $M$ under this equivalence and let $A$ be the dg endomorphism algebra of $\widetilde{M}$.
Then there is a natural triangle functor $\ck_{\rm ac}(\proj\ce)\to\per(A)$ which turns out to be a triangle equivalence and takes  $\widetilde{M}$ to $A_A$.
Composing this equivalence with the equivalence $\ck_{\rm ac}(\proj\ce)\rightarrow \ct$, we obtain a triangle equivalence $\ct\to\per(A)$  which takes $M$ to $A_A$.

\section{Silting reduction as subfactor category}\label{s:subfactor-category}

A silting reduction of a triangulated category $\ct$ was introduced in \cite{AI} as the triangle quotient $\ct/\thick\cp$ of $\ct$ by the thick subcategory $\thick\cp$ generated by a presilting subcategory $\cp$ of $\ct$.
In this section we show that under mild conditions, the silting reduction of $\ct$ can be realized as a certain subfactor category of $\ct$.
Moreover we show that there is a bijection between silting subcategories of $\ct$ containing $\cp$ and silting subcategories of the silting reduction  $\ct/\thick\cp$.
We also discuss various applications of this result.

\subsection{The additive equivalence}\label{ss:add-equiv}

Let $\ct$ be a triangulated category.
We fix a presilting subcategory $\cp$ of $\ct$. 
Let
\[\SS:=\thick_{\ct}\cp\ \mbox{ and }\ \cu:=\ct/\SS.\]
We call $\cu$ the \emph{silting reduction} of $\ct$ with respect to $\cp$ (see \cite{AI}). We refer to \cite{Neeman01b} for the standard description of morphisms in triangle quotient categories, which are heavily used in this section and Section~\ref{s:silting-red-vs-cy-red}. 
 In the rest, we assume $\cp=\add\cp$ for simplicity. 
For an integer $\ell$, there is a bounded co-t-structure $(\SS_{\ge\ell},\SS_{\le\ell})$ on $\SS$ by Proposition~\ref{from silting to co-t-structure}, where
\begin{eqnarray*}
\SS_{\ge\ell}=\SS_{>\ell-1}&:=&\bigcup_{i\ge0}\cp[-\ell-i]*\cdots*\cp[-\ell-1]*\cp[-\ell],\\
\SS_{\le\ell}=\SS_{<\ell+1}&:=&\bigcup_{i\ge0}\cp[-\ell]*\cp[-\ell+1]*\cdots*\cp[-\ell+i].
\end{eqnarray*} 
We introduce a full subcategory $\cz$ of $\ct$ by
\[\cz:=({}^{\perp_\ct}\SS_{<0})\cap(\SS_{>0}{}^{\perp_\ct})
=({}^{\perp_\ct}\cp[>\hspace{-3pt}0])\cap(\cp[<\hspace{-3pt}0]{}^{\perp_\ct}).\]
Since $\cp$ is presilting, we have $\cp\subset \cz$.

Now we consider the following mild technical conditions:
\begin{itemize}
\item[(P1)] $\cp$ is covariantly finite in ${}^{\perp_\ct}\SS_{<0}$ and contravariantly finite in $\SS_{>0}{}^{\perp_\ct}$.
\item[(P2)] For any $X\in\ct$, we have $\Hom_{\ct}(X,\cp[\ell])=0=\Hom_\ct(\cp,X[\ell])$ for $\ell\gg0$.
\end{itemize}
For example, (P1) is satisfied when $\ct$ is Hom-finite over a field and $\cp=\add(P)$ for a presilting object $P$; by Lemma \ref{vanishing for >>}, (P2) is satisfied when $\ct$ admits a silting subcategory which contains $\cp$.

The following result shows that we can realise the triangle quotient $\cu=\ct/\cs$ as a subfactor category of $\ct$. Let $\rho\colon\ct\to\cu$ be the canonical projection functor.

\begin{theorem}\label{equivalence}
Under the conditions (P1) and (P2), the composition $\cz\subset\ct\xrightarrow{\rho}\cu$ of natural functors induces an equivalence of additive categories:
\[\bar{\rho}\colon\frac{\cz}{[\cp]}\stackrel{\simeq}{\longrightarrow}\cu.\]
\end{theorem}


The rest of this subsection is devoted to the proof of Theorem~\ref{equivalence}. 
Since $\rho(\cp)=0$, the composition $\cz\subset\ct\xrightarrow{\rho}\cu$ induces a functor $\bar{\rho}\colon\frac{\cz}{[\cp]}\to\cu$.
To prove that this is an equivalence, we start with the following useful observation, which generalises Proposition~\ref{from silting to co-t-structure}.

\begin{proposition}\label{torsion pairs}
The following conditions are equivalent.
\begin{itemize}
\item[(a)] The conditions (P1) and (P2) are satisfied.
\item[(b)] The two pairs $({}^{\perp_\ct}\SS_{<0},\SS_{\leq 0})$ and $(\SS_{\geq 0},\SS_{>0}{}^{\perp_\ct})$ are co-t-structures on $\ct$.
\end{itemize}
In this case, the co-hearts of these co-t-structures are $\cp$.
\end{proposition}

\begin{proof} 
First, we prove $\cp=({}^{\perp_\ct}\SS_{<0})\cap\SS_{\leq 0}=\SS_{\geq 0}\cap(\SS_{>0}{}^{\perp_\ct})$. We only prove the first equality since the second one is dual. It suffices to show that any $X\in({}^{\perp_\ct}\SS_{<0})\cap\SS_{\leq 0}$ belongs to $\cp$. Since $\cs_{\le0}=\cp*\cs_{<0}$, we get $X\in\add\cp=\cp$ by the dual of Lemma~\ref{belong to add}.

(a)$\Rightarrow$(b)
We only prove that $({}^{\perp_\ct}\SS_{<0},\SS_{\leq 0})$ is a co-t-structure on $\ct$ since the other assertion can be shown similarly.
This is equivalent to showing that $({}^{\perp_\ct}\SS_{<0},\SS_{<0})$ is a torsion pair.
Since ${}^{\perp_\ct}\SS_{<0}=\add{}^{\perp_\ct}\SS_{<0}$ holds and $\SS_{\leq 0}=\add\SS_{\leq 0}$ holds by Proposition \ref{additively closed},
it is enough to show that any object $X\in\ct$ belongs to $({}^{\perp_\ct}\SS_{<0})*\SS_{<0}$.
By our assumption (P2), there exists some integer $\ell$ such that $X\in{}^{\perp_\ct}\SS_{<-\ell}$.
If $\ell\le0$, then ${}^{\perp_\ct}\SS_{<-\ell}\subset {}^{\perp_\ct}\SS_{<0}$ and the assertion follows.
Thus we assume $\ell>0$ and induct on $\ell$.
By our assumption (P1), there exists a triangle
\[
\xymatrix{
Y\ar[r] & X\ar[r]^{f} & P[\ell]\ar[r]& Y[1]
}
\]
with a left $\cp[\ell]$-approximation $f$ of $X$.
Applying $\Hom_{\ct}(-,\SS_{<-\ell})$ and $\Hom_{\ct}(-,\cp[\ell])$, we have $Y\in{}^{\perp_\ct}\SS_{\le-\ell}$.
By the induction hypothesis, we have $Y\in({}^{\perp_\ct}\SS_{<0})*\SS_{<0}$.
Thus $X\in Y*P[\ell]\in({}^{\perp_\ct}\SS_{<0})*(\SS_{<0}*\cp[\ell])=({}^{\perp_\ct}\SS_{<0})*\SS_{<0}$ holds since $\cs_{<0}$ is extension closed.

(b)$\Rightarrow$(a) For any $X\in{}^{\perp_\ct}\SS_{<0}$, take a triangle $Y\to X\xrightarrow{a} X_{\le0}\to Y[1]$ with $Y\in{}^{\perp_\ct}\SS_{\le0}$ and $X_{\le0}\in\cs_{\le0}$. Then $X_{\le0}$ belongs to $({}^{\perp_\ct}\SS_{<0})*({}^{\perp_\ct}\SS_{<0})={}^{\perp_\ct}\SS_{<0}$ and hence to $({}^{\perp_\ct}\SS_{<0})\cap\SS_{\le0}=\cp$.
Since $\Hom_{\ct}(Y,\cp)=0$, it follows that $a$ is a left $\cp$-approximation. Thus $\cp$ is covariantly finite in ${}^{\perp_\ct}\SS_{<0}$. Dually, $\cp$ is contravariantly finite in $\SS_{>0}{}^{\perp_\ct}$.

By the definition of  ${}^{\perp_\ct}\SS_{<0}$ we have $\Hom_\ct({}^{\perp_\ct}\SS_{<0},\cp[>\hspace{-3pt}0])=0$.
For any $X$ in $\cs$, $\Hom_\ct(X,\cp[\gg\hspace{-3pt}0])=0$ holds.
Since any $X$ in $\ct$ belongs to $({}^{\perp_\ct}\SS_{<0})*\SS_{<0}$, we have $\Hom_\ct(X,\cp[\gg\hspace{-3pt}0]) = 0$.
Dually, we have $\Hom_\ct(\cp,X[\gg\hspace{-3pt}0])=0$. Thus (P2) holds.
\end{proof}

Next we show that our functor in Theorem \ref{equivalence} is dense.

\begin{lemma}
For any $X\in\ct$, there exists $Y\in\cz$ satisfying $X\simeq Y$ in $\cu$.
As a consequence, the  functor $\bar{\rho}:\frac{\cz}{[\cp]}\to\cu$ in Theorem~\ref{equivalence} is dense.
\end{lemma}

\begin{proof} Let $X\in\cu$. 
By Proposition \ref{torsion pairs}, we have a triangle
\[
\xymatrix{
X'\ar[r] & X\ar[r] & S\ar[r] & X'[1] & (X'\in{}^{\perp_\ct}\SS_{<0},\ S\in\SS_{<0}).
}
\]
Then we have $X\simeq X'$ in $\cu$.
Again by Proposition \ref{torsion pairs}, we have a triangle
\[
\xymatrix{
S'\ar[r] & X'\ar[r] & Y\ar[r] & S'[1] & (S'\in\SS_{>0},\ Y\in\SS_{>0}{}^{\perp_\ct}).
}
\]
Then we have $X\simeq X'\simeq Y$ in $\cu$.
Applying $\Hom_{\ct}(-,\SS_{<0})$, we see that
$\Hom_{\ct}(Y,\SS_{<0})\simeq\Hom_{\ct}(X',\SS_{<0})$ vanishes.
Thus $Y$ belongs to $({}^{\perp_\ct}\SS_{<0})\cap(\SS_{>0}{}^{\perp_\ct})=\cz$, and we have an isomorphism $X\simeq Y$ in $\cu$.
\end{proof}

Finally we show that our functor is fully faithful.

\begin{lemma}\label{fully faithful}
The functor $\rho\colon\ct\to\cu$ induces the following bijective maps for any $M\in{}^{\perp_\ct}\SS_{<0}$ and $N\in\SS_{>0}{}^{\perp_\ct}$
\begin{align*}
\Hom_{\frac{\ct}{[\cp]}}(M,N)&\longrightarrow\Hom_{\cu}(M,N),\\
\Hom_{\ct}(M,N[\ell])&\longrightarrow\Hom_{\cu}(M,N[\ell])\qquad(\ell>0).
\end{align*}
As a consequence, the functor $\bar{\rho}:\frac{\cz}{[\cp]}\to\cu$ in Theorem~\ref{equivalence} is fully faithful.
\end{lemma}

\begin{proof}
We first show the surjectivity. 

Let $\ell\ge0$. Any morphism in $\Hom_{\cu}(M,N[\ell])$ has a representative of the form $M\xrightarrow{f}X\xleftarrow{s}N[\ell]$, where $f\in\Hom_{\ct}(M,X)$ and $s\in\Hom_{\ct}(N[\ell],X)$, such that the cone of $s$ is in $\SS$.
Take a triangle
\[
\xymatrix{
N[\ell]\ar[r]^{s} & X\ar[r] & S\ar[r]^(0.4){a} & N[\ell+1] & (S\in\SS).
}
\]
By Proposition~\ref{from silting to co-t-structure}, we can take a triangle
\[
\xymatrix{
S_{\ge0}\ar[r]^{b} &S\ar[r] &S_{<0}\ar[r] & S_{\ge0}[1] & (S_{\ge0}\in\SS_{\ge0},\ S_{<0}\in\SS_{<0}).
}
\]
Since $ab=0$ by $S_{\ge0}\in\SS_{\ge0}$ and $N[\ell+1]\in\SS_{>-\ell-1}{}^{\perp_\ct}$, we have the following commutative diagram by the octahedral axiom.
\[\xymatrix{
&S_{\ge0}\ar@{=}[r]\ar[d]&S_{\ge0}\ar[d]^b\\
N[\ell]\ar[r]^s\ar@{=}[d]&X\ar[r]\ar[d]^c&S\ar[r]^a\ar[d]&N[\ell+1]\ar@{=}[d]\\
N[\ell]\ar[r]^{cs}&X'\ar[r]^d\ar[d]&S_{<0}\ar[r]\ar[d]&N[\ell+1]\\
&S_{\ge0}[1]\ar@{=}[r]&S_{\ge0}[1]
}\]
Then we have $dcf=0$ by $M\in{}^{\perp_\ct}\SS_{<0}$ and $S_{<0}\in\SS_{<0}$.
Thus there exists $e\in\Hom_{\ct}(M,N[\ell])$ such that $cf=cse$.
Now $c(f-se)=0$ implies that $f-se$ factors through $S_{\ge0}\in\cs$.
Thus $f=se$ and $s^{-1}f=e$ hold in $\cu$, and we have the assertion.

\smallskip
Next we show the injectivity.

Let $\ell\ge0$. Assume that a morphism $f\in\Hom_{\ct}(M,N[\ell])$ is zero in $\cu$. 
Then it factors through $\cs$ (by, for example, \cite[Lemma 2.1.26]{Neeman01b}), that is, there exist $S\in\cs$, $g\in\Hom_{\ct}(M,S)$ and
$a\in\Hom_{\ct}(S,N[\ell])$ such that $f=ag$.
Take a triangle
\[
\xymatrix{
S_{>-\ell}\ar[r]^(0.55){b} & S\ar[r]^(0.4){c} & S_{\le-\ell}\ar[r] & S_{>-\ell}[1] & (S_{>-\ell}\in\SS_{>-\ell},\ S_{\le-\ell}\in\SS_{\le-\ell}).
}
\]
Since $ab=0$ by $S_{>-\ell}\in\SS_{>-\ell}$ and $N[\ell]\in\SS_{>-\ell}{}^{\perp_\ct}$, there exists $d\in\Hom_{\ct}(S_{\le-\ell},N[\ell])$ such that $a=dc$.
\[\xymatrix{
S_{>-\ell}\ar[r]^(0.55){b} & S\ar[r]^(0.4){c}\ar[dr]_a & S_{\le-\ell}\ar[d]^d\\
M\ar[rr]_f\ar[ru]_g&&N[\ell]
}
\]
First we assume $\ell>0$. Then $cg=0$ because $M\in{}^{\perp_\ct}\SS_{<0}$ and $S_{\le-\ell}\in\SS_{\le-\ell}\subset \cs_{<0}$. Thus we have $f=dcg=0$.

Next we assume $\ell=0$. Take a triangle
\[
\xymatrix{
P\ar[r] & S_{\le0}\ar[r]^(0.47){e} & S_{<0}\ar[r] & P[1] & (P\in\cp,\ S_{<0}\in\SS_{<0}).
}
\]
Then we have $ecg=0$ by $M\in{}^{\perp_\ct}\SS_{<0}$ and $S_{<0}\in\SS_{<0}$.
Thus $cg$ factors through $P$, and $f=dcg=0$ in $\frac{\ct}{[\cp]}$.
\end{proof}

\subsection{The triangle equivalence}\label{subsection: Triangle equivalence}

Let $\ct$ be a triangulated category and $\cp$ a presilting subcategory of $\ct$ satisfying (P1) and (P2). Keep the notation in Section~\ref{ss:add-equiv}.
The aim of this subsection is to show that the additive category
$\frac{\cz}{[\cp]}$ has the structure of a triangulated category,
and that the equivalence given in Theorem
\ref{equivalence} is a triangle equivalence.

\begin{lemma}\label{mutation pair}
The pair $(\cz,\cz)$ forms a $\cp$-mutation pair (see Section~\ref{section: Mutation pair}). More precisely,
 for $T\in\ct$, the following conditions are equivalent.
\begin{itemize}
\item[(a)] $T\in\cz$.
\item[(b)] There exists a triangle $X\xrightarrow{a}P\to T\to X[1]$ with $X\in\cz$ and a left $\cp$-approximation $a$.
\item[(c)] There exists a triangle $T\to P'\xrightarrow{b}Y\to T[1]$ with $Y\in\cz$ and a right $\cp$-approximation $b$.
\end{itemize}
\end{lemma}

\begin{proof}
We only show the equivalence of (a) and (b) since the equivalence
of (a) and (c) can be shown dually.

(b)$\Rightarrow$(a) By applying $\Hom_{\ct}(\cp,-)$ to the triangle, we obtain $\Hom_{\ct}(\cp,T[>\hspace{-3pt}0])=0$.
Similarly by applying $\Hom_{\ct}(-,\cp)$ to the triangle, we obtain $\Hom_{\ct}(T,\cp[>\hspace{-3pt}0])=0$. Therefore $T\in\cz$.

(a)$\Rightarrow$(b) By (P1), there exists a triangle $X\xrightarrow{a} P\xrightarrow{b}T\to X[1]$
with a right $\cp$-approximation $b$.
By applying $\Hom_{\ct}(\cp,-)$ to the triangle, we obtain $\Hom_{\ct}(\cp,X[>\hspace{-3pt}0])=0$.
Similarly by applying $\Hom_{\ct}(-,\cp)$ to the triangle we obtain that $\Hom_{\ct}(X,\cp[>\hspace{-3pt}0])=0$ holds and that $a$ is a left $\cp$-approximation. Therefore $X\in\cz$.
\end{proof}

As a consequence of this lemma, the category $\frac{\cz}{[\cp]}$ has the natural structure of a triangulated category, according to Theorem~\ref{triangle structure of Z/P}.
Now we prove the following result.


\begin{theorem}\label{t:triangle-equivalence}
The category $\frac{\cz}{[\cp]}$ has a structure of a triangulated category given in Theorem~\ref{triangle structure of Z/P} such that
the functor $\bar{\rho}\colon\frac{\cz}{[\cp]}\to\cu$ in Theorem~\ref{equivalence} is a triangle equivalence.
\end{theorem}

\begin{proof}
We need to show that the equivalence $\bar{\rho}\colon\frac{\cz}{[\cp]}\rightarrow\cu$ is a triangle functor.

Applying the triangle functor $\rho$ to the triangle
$X\to P_X\to X\langle1\rangle\to X[1]$ in Theorem~\ref{triangle structure of Z/P}(a), we have an isomorphism
$X\langle1\rangle\to X[1]$ in $\cu$, which defines a natural isomorphism $\bar{\rho}\circ\langle 1\rangle\simeq [1]\circ\bar{\rho}$.

Let
\begin{equation}\label{triangle in U}
\xymatrix{
X\ar[r]^{\overline{f}} & Y\ar[r]^{\overline{g}} & Z 
\ar[r]^{\overline{a}} & X\langle1\rangle
}
\end{equation}
be a triangle given in Theorem~\ref{triangle structure of Z/P}(b).
Applying the triangle functor $\ct\to\cu$ to \eqref{define a triangle},
we have a commutative diagram
\[
\xymatrix{
X\ar[r]^f\ar@{=}[d]&Y\ar[r]^g\ar[d]&Z\ar[r]^h\ar[d]^a&X[1]\ar@{=}[d]\\
X\ar[r]&0\ar[r]&X\langle1\rangle\ar[r]^\sim&X[1]
}
\]
of triangles in $\cu$. Thus the image of \eqref{triangle in U}
by the functor $\frac{\cz}{[\cp]}\to\cu$ is a triangle.
\end{proof}

We remark that more general versions of Theorems~\ref{equivalence} and~\ref{t:triangle-equivalence} have since been established in \cite{Wei15,Li16,N,IY}.

\subsection{The correspondence between silting subcategories}

Let $\ct$ be a triangulated category.
Recall that $\silt\ct$ (respectively, $\presilt\ct$) is the class of silting (respectively, presilting) subcategories of $\ct$, where we identify two (pre)silting subcategories $\MM$ and $\NN$ of $\ct$ when $\add\MM=\add\NN$.

Fix a presilting subcategory $\cp$ of $\ct$ and denote by $\silt_{\cp}\ct$ (respectively, $\presilt_{\cp}\ct$) the class of silting (respectively, presilting) subcategories of $\ct$ containing $\cp$. Assume further that the conditions (P1) and (P2) are satisfied. Keep the notation in Section~\ref{ss:add-equiv}.

\begin{theorem}\label{t:silting-reduction}
The natural functor $\rho\colon\ct\to\cu$ induces bijections $\silt_{\cp}\ct\to\silt\cu$ and $\presilt_{\cp}\ct\to\presilt\cu$.
\end{theorem}

\begin{proof}
(i) We will show that $\rho$ induces a map $\presilt_{\cp}\ct\to\presilt\cu$.

Let $\MM$ be a presilting subcategory of $\ct$ containing $\cp$.
Then we have $\MM\subset\cz$.
By Lemma \ref{fully faithful}, we have
\[\Hom_{\cu}(\MM,\MM[>\hspace{-3pt}0])=\Hom_{\ct}(\MM,\MM[>\hspace{-3pt}0])=0.\]
Thus $\rho(\MM)$ is a presilting subcategory of $\cu$.

(ii) We will show that the map $\presilt_{\cp}\ct\to\presilt\cu$ is bijective.

Since $\rho$ induces an equivalence $\frac{\cz}{[\cp]}\simeq\cu$, the correspondence $\presilt_{\cp}\ct\to\presilt\cu$ is injective.
We will show the surjectivity.
For a presilting subcategory $\NN$ of $\cu$, we define a subcategory $\MM$ of $\ct$ by
\[\MM:=\{X\in\cz\ |\ \rho(X)\in\NN\}.\]
Then $\cp\subset\cm$ and $\rho(\MM)=\NN$ hold. Moreover, by Lemma \ref{fully faithful}, we have
\[\Hom_{\ct}(\MM,\MM[>\hspace{-3pt}0])=\Hom_{\cu}(\NN,\NN[>\hspace{-3pt}0])=0.\]
Thus $\cm\in\presilt_{\cp}\ct$ holds, and the assertion follows.

(iii) We will show that $\rho$ induces a bijective map $\silt_{\cp}\ct\to\silt\cu$.

Let $\MM$ be a presilting subcategory of $\ct$ contaning $\cp$ and
$\NN:=\rho(\MM)$ the corresponding presilting subcategory of $\cu$.
By (ii), it is enough to show that $\thick_{\ct}\MM=\ct$ holds
if and only if $\thick_{\cu}\NN=\cu$ holds. This follows from the fact that
$\rho$ induces a bijection between thick subcategories of $\ct$ containing $\cp$ and thick subcategories of $\cu$ (\cite[Proposition 2.3.1 (c)${}^{\mathrm{bis}}$ (d)${}^{\mathrm{bis}}$]{Verdier96}).
\end{proof}

Moreover the bijection above is compatible with the natural partial order defined in Section~\ref{ss:silting-and-t-structure}.

\begin{corollary}\label{iso of posets}
The bijection $\silt_{\cp}\ct\to\silt\cu$ in Theorem \ref{t:silting-reduction} is an isomorphism of partially ordered sets.
\end{corollary}

\begin{proof}
Let $\MM$ and $\NN$ be silting subcategories of $\ct$ containing $\cp$.
Then $\MM\subset\cz$ and $\NN\subset\cz$ hold.
By Lemma \ref{fully faithful}, we have
\[\Hom_{\ct}(\MM,\NN[>\hspace{-3pt}0])\simeq\Hom_{\cu}(\MM,\NN[>\hspace{-3pt}0]).\]
Thus $\MM\ge\NN$ if and only if $\rho(\MM)\ge\rho(\NN)$.
\end{proof}

Next we discuss the completion of ``almost complete" presilting subcategories.

Let $\cm$ be a silting subcategory of $\ct$ containing $\cp$. Then $\cm\subset\cz$ and hence $\cp$ is functorially finite in $\cm$ by (P1), and therefore each $X\in\cm$ admits triangles
\[\xymatrix{X\ar[r]^f&P'\ar[r]&Y_X\ar[r]&X[1]}\ \mbox{ and }\ \xymatrix{Z_X\ar[r]&P''\ar[r]^g&X\ar[r]&Z[1]}\]
in $\ct$ with a left $\cp$-approximation $f$ of $X$ and a right $\cp$-approximation $g$ of $X$. It was shown in \cite[Theorem 2.31]{AI} that
\[\mu^-_{\cp}(\cm):=\add(\cp\cup\{Y_X\mid X\in\cm\})\ \mbox{ and }\ \mu^+_{\cp}(\cm):=\add(\cp\cup\{Z_X\mid X\in\cm\})\]
are again silting subcategories of $\ct$, which we call \emph{the left mutation and the right mutation of $\cm$ at $\cp$}, respectively. Moreover the maps
\[\mu^-_{\cp}\colon\silt_\cp\ct\to\silt_\cp\ct\ \mbox{  and }\ \mu^+_{\cp}\colon\silt_\cp\ct\to\silt_\cp\ct\]
are mutually inverse \cite[Proposition 2.33]{AI}.

The following result was shown in \cite[Theorem 2.44]{AI} under the strong restriction that $\thick\cp$ is functorially finite in $\ct$.

\begin{corollary}
Assume that $\ct$ is Krull--Schmidt.
Assume that there exists an indecomposable object $X_0\in\ct$ such that $X_0\notin\cp$ and $\cm:=\add(\cp\cup\{X_0\})$ is a silting subcategory of $\ct$.
Then we have
\[\silt_\cp\ct=\{\mu^{+i}_\cp(\cm),\ \cm,\ \mu^{-i}_\cp(\cm)\mid i>0\}.\]
\end{corollary}

\begin{proof}
By construction, there exists an indecomposable object $X_i\in\ct$ for any $i\in\Z$ such that $\mu^{+i}_\cp(\cm)=\add(\cp\cup\{X_i\})$ and $\mu^{-i}_\cp(\cm)=\add(\cp\cup\{X_{-i}\})$ for any $i>0$. Then $X_i=X_0\langle i\rangle$ holds by our construction.
By Theorem \ref{t:silting-reduction}, we have a bijection $\silt_{\cp}\ct\to\silt\cu$. In particular $\cu$ has an indecomposable silting object $X_0$. By \cite[Theorem 2.26]{AI}, we have $\silt\cu=\{X_0\langle i\rangle\mid i\in\Z\}$. Therefore $\silt_\cp\ct$ has the desired description. 
\end{proof}

\subsection{A theorem of Buchweitz}

Recall that a noetherian ring $A$ is called \emph{Iwanaga--Gorenstein} if $A$ has finite injective dimension as an $A$-module and also as an $A^{\op}$-module (see e.g. \cite{EJ}).
In this case, we define the category of \emph{Cohen--Macaulay $A$-modules} (also often called \emph{modules of Gorenstein dimension zero}, \emph{Gorenstein projective modules}, or \emph{totally reflexive modules}) by
\[\CM A:=\{X\in\mod A\mid\Ext_A^i(X,A)=0\ \mbox{ for any $i>0$}\}.\]
This has a natural structure of a Frobenius category whose projective-injective objects are exactly the projective $A$-modules, and we denote by $\underline{\CM} A$ its stable category. We recover the following classical result  due to Buchweitz as a consequence of Theorem~\ref{t:triangle-equivalence}.

\begin{theorem}[{\cite[Theorem 4.4.1(b)]{Buchweitz87}}]\label{buchweitz}
Let $A$ be an Iwanaga--Gorenstein ring. Then
\[\CM A=\{X\in\cd^{\bo}(\mod A)\mid\Hom_{\cd^{\bo}(\mod A)}(X,A[>\hspace{-3pt}0])=0=\Hom_{\cd^{\bo}(\mod A)}(A[<\hspace{-3pt}0],X)\}\]
holds, and the embedding $\CM A\to\mod A$ induces a triangle equivalence
\[\underline{\CM} A\xrightarrow{\simeq}\cd^{\rm b}(\mod A)/\ck^{\rm b}(\proj A).\]
\end{theorem}

To prove this, we need the following duality (see \cite{Hartshorne} for the commutative case).

\begin{lemma}[{\cite[Corollary 2.11]{Miyachi}}]\label{* dual}
Let $A$ be an Iwanaga-Gorenstein ring. Then we have a duality
$(-)^*=\RHom_A(-,A)\colon\cd^{\bo}(\mod A)\to\cd^{\bo}(\mod A^{\op})$, which has a quasi-inverse $(-)^*=\RHom_{A^{\op}}(-,A)\colon\cd^{\bo}(\mod A^{\op})\to\cd^{\bo}(\mod A)$.
\end{lemma}

\begin{proof}[Proof of Theorem~\ref{buchweitz}]
It suffices to prove the first equality. In fact, let $\ct:=\cd^{\rm b}(\mod A)$, $\cp:=\proj A$ and $\cs:=\ck^{\bo}(\proj A)$. Then
$\cz=({}^{\perp_\ct}\SS_{<0})\cap(\SS_{>0}{}^{\perp_\ct})$ is the right hand side of the desired equality. Thus $\frac{\cz}{[\cp]}=\underline{\CM} A$ holds, and by Theorem~\ref{t:triangle-equivalence} we obtain the first triangle equivalence.

Let $(\cd^{\le0}(\mod B),\cd^{\ge0}(\mod B))$ be the standard t-structure on $\cd^{\rm b}(\mod B)$ for $B=A$ or $A^{\op}$.
Let $\ct':=\cd^{\rm b}(\mod A^{\op})$, $\cp':=\proj A^{\op}$ and $\cs':=\ck^{\bo}(\proj A^{\op})$. Then we have
\begin{equation}\label{>0 calculation}
\SS_{>0}{}^{\perp_\ct}=A[<\hspace{-3pt}0]^{\perp_\ct}=\cd^{\le0}(\mod A)\ \mbox{ and }\ \SS'_{>0}{}^{\perp_{\ct'}}=A[<\hspace{-3pt}0]^{\perp_{\ct'}}=\cd^{\le0}(\mod A^{\op}).
\end{equation}
In particular, we have $\mod A\subset \SS_{>0}{}^{\perp_\ct}$ and
\[\mod A\cap\cz=\mod A\cap({}^{\perp_\ct}\SS_{<0})=\CM A.\]
It is enough to show $\cz\subset\mod A$.
By the duality in Lemma~\ref{* dual}, we have $\cs_{<0}=(\cs'_{>0})^*$ and
\[{}^{\perp_\ct}\SS_{<0}=(\SS'_{>0}{}^{\perp_{\ct'}})^*\stackrel{\eqref{>0 calculation}}{=}(\cd^{\le0}(\mod A^{\op}))^*\subset\cd^{\ge0}(\mod A).\]
Therefore
$\cz=({}^{\perp_\ct}\SS_{<0})\cap(\SS_{>0}{}^{\perp_\ct})\subset\cd^{\le0}(\mod A)\cap\cd^{\ge0}(\mod A)=\mod A$ holds.
\end{proof}

Another application of Theorem~\ref{equivalence} is the following.

\begin{corollary}\label{cor:relative-singularity-cat}
Let $k$ be a field and $A$ be a finite-dimensional $k$-algebra. Assume that $P$ is a finitely generated projective $A$-module which has finite injective dimension. Then the triangle quotient $\cd^{\rm b}(\mod A)/\thick P$ is Hom-finite and Krull--Schmidt. 
\end{corollary}
\begin{proof}
Let $\cp=\add P$. Then (P1) is automatically satisfied. Thanks to the assumption that $P$ is projective of finite injective dimension, (P2) is also satisfied. Define the full subcategory $\cz$ of $\cd^{\rm b}(\mod A)$ as in Section~\ref{ss:add-equiv}. Then $\cz$ is closed under direct summands. Thus it is Hom-finite and Krull--Schmidt, so is the additive quotient $\frac{\cz}{[\cp]}\simeq \cd^{\rm b}(\mod A)/\thick P$. 
\end{proof}

As an application of Corollary \ref{cor:relative-singularity-cat}, it follows that for a finite-dimensional $k$-algebra $A$ which is right Iwanaga--Gorenstein, i.e. $A_A$ has finite injective dimension, the singularity category $\cd_{\mathrm{sg}}(A)=\cd^{\bo}(\mod A)/\ck^{\bo}(\proj A)$ is Hom-finite and Krull--Schmidt.

\section{t-structures adjacent to silting subcategories}\label{section:adjacent t-structures}

The aim of this section is to show that silting subcategories always yield co-t-structures, and under certain conditions they also yield t-structures.
We refer to \cite{KellerVossieck88,HKM,BeligiannisReiten07,AI,KellerNicolas11,KY,AngeleriMarksVitoria14,PsaroudakisVitoria14} for related results on this subject.
In particular, results in this section will play an important role in Section~\ref{s:silting-red-vs-cy-red}.

Let $\ct$ be a triangulated category. 
For a silting subcategory $\cm$ in $\ct$ satisfying $\cm=\add(\cm)$, we have a co-t-structure $(\ct_{\ge0},\ct_{\le0})$ on $\ct$ by Proposition~\ref{from silting to co-t-structure}, where
\begin{eqnarray*}
\ct_{\ge0}=\bigcup_{n\ge0}\cm[-n]*\cdots*\cm[-1]*\cm\ \mbox{ and }\ 
\ct_{\le0}=\bigcup_{n\ge0}\cm*\cm[1]*\cdots*\cm[n].
\end{eqnarray*}
Now we consider the pair $(\cm[<\hspace{-3pt}0]^{\perp_{\ct}},\cm[>\hspace{-3pt}0]^{\perp_{\ct}})$, where
\begin{eqnarray*}
\cm[<\hspace{-3pt}0]^{\perp_{\ct}}&=&\{X\in\ct\mid\Hom_{\ct}(\cm[<\hspace{-3pt}0],X)=0\},\\
\cm[>\hspace{-3pt}0]^{\perp_{\ct}}&=&\{X\in\ct\mid\Hom_{\ct}(\cm[>\hspace{-3pt}0],X)=0\}.
\end{eqnarray*}
We have the following immediate observations.

\begin{lemma}\label{no Homs}
We have $\cm[<\hspace{-3pt}0]^{\perp_{\ct}}=\ct_{\le0}$ and $\Hom_{\ct}(\cm[<\hspace{-3pt}0]^{\perp_{\ct}},\cm[>\hspace{-3pt}0]^{\perp_{\ct}}[-1])=0$.
\end{lemma}

\begin{proof}
By Proposition \ref{from silting to co-t-structure} (a), we have $\cm[<\hspace{-3pt}0]^{\perp_{\ct}}=\ct_{\ge1}{}^{\perp_{\ct}}=\ct_{\le 0}$.
The vanishing of Hom-spaces is then a direct consequence.
\end{proof}

Following Bondarko \cite{Bondarko10}, we say that \emph{$\cm$ has a right adjacent t-structure} if $(\cm[<\hspace{-3pt}0]^{\perp_{\ct}},\cm[>\hspace{-3pt}0]^{\perp_{\ct}})$ forms a t-structure on $\ct$.
By Lemma~\ref{no Homs}, this is equivalent to that $\ct=\cm[<\hspace{-3pt}0]^{\perp_{\ct}}*\cm[\ge\hspace{-3pt}0]^{\perp_{\ct}}$ holds.
Dually, we say that \emph{$\cm$ has a left adjacent t-structure} if $({}^{\perp_{\ct}}\cm[<\hspace{-3pt}0],{}^{\perp_{\ct}}\cm[>\hspace{-3pt}0])$
is a t-structure on $\ct$. Note that we have dual version of Lemma~\ref{no Homs}.

\begin{proposition}\label{p:adjacent-->finiteness}
If $\cm$ has a right (respectively, left) adjacent t-structure, then it is a contravariantly finite (respectively, covariantly finite) subcategory of $\ct$.
\end{proposition}
\begin{proof}
We only prove the statement for right adjacent t-structures.
Because $(\cm[<\hspace{-3pt}0]^{\perp_{\ct}},\cm[>\hspace{-3pt}0]^{\perp_{\ct}})$ is a t-structure, $\cm[<\hspace{-3pt}0]^{\perp_{\ct}}$ is a contravariantly finite subcategory of $\ct$. 
It is enough to show that any $X\in\cm[<\hspace{-3pt}0]^{\perp_{\ct}}$ has a right $\cm$-approximation. There exists a triangle
$M\xrightarrow{f} X\to Y\to M[1]$ with $M\in\cm$ and $Y\in\cm[\le\hspace{-3pt}0]^{\perp_{\ct}}$ by Proposition \ref{from silting to co-t-structure}(a), from which it follows that $f$ is a right $\cm$-approximation, giving the claim.
\end{proof}

\subsection{Compatible silting subcategories}

In this subsection, we prove  that the property of having an adjacent t-structure is invariant under a suitable change of silting subcategories.
We say that two silting subcategories $\cm$ and $\cn$ of $\ct$ are
\emph{compatible} if there exist integers $\ell,\ell' > 0$ such that
$\cm[-\ell']\ge\cn\ge\cm[\ell]$, or equivalently, $\cn[-\ell]\ge\cm\ge\cn[\ell']$. By Proposition~\ref{from silting to co-t-structure}(b), these two conditions are  equivalent to the following two conditions, respectively,
\begin{eqnarray*}
\cn&\subset&\cm[-\ell']*\cm[1-\ell]*\cdots*\cm[\ell-1]*\cm[\ell],\\
\cm&\subset&\cn[-\ell]*\cn[1-\ell]*\cdots*\cn[\ell-1]*\cn[\ell'].
\end{eqnarray*}
Compatibility gives an equivalence relation on $\silt\ct$.

\begin{theorem}\label{l:shift-t-str}
Let $\ct$ be a triangulated category, and $\cm$ and $\cn$ contravariantly finite (respectively, covariantly finite) silting subcategories of $\ct$ which are compatible with each other.
Then $\cm$ has a right (respectively, left) adjacent t-structure if and only if $\cn$ has a right (respectively, left) adjacent t-structure.
\end{theorem}

Since all silting objects in $\ct$ are compatible with each other,  we obtain the following special case.
 
\begin{theorem}\label{l:shift-t-str 1}
Let $\ct$ be a triangulated category satisfying the condition (F) given in Section~\ref{ss:notation}, and $M$ and $N$ silting objects of $\ct$.
Then $M$ has a right (respectively, left) adjacent t-structure if and only if $N$ has a right (respectively, left) adjacent t-structure.
\end{theorem}

We start the proof of Theorem~\ref{l:shift-t-str} with the following easy observations.

\begin{lemma}\label{basic right and left}
Let $\ct$ be a triangulated category.
\begin{itemize}
\item[(a)] The opposite category $\ct^{\op}$ of $\ct$ has a
natural structure of a triangulated category.
\item[(b)] There is a bijection $\silt\ct\to\silt\ct^{\op}$ given by $\cm\mapsto\cm^{\op}$.
\item[(c)] $\cm$ has a left adjacent t-structure in $\ct$ if and only if $\cm^{\op}$ has a right adjacent t-structure in $\ct^{\op}$.
\end{itemize}
\end{lemma}

\begin{proof}[Proof of Theorem \ref{l:shift-t-str}.]
By Lemma \ref{basic right and left}, we only have to prove the statement for right adjacent t-structures.
We will prove the `only if' part, that is, if $\cm$ has a right adjacent t-structure, then $\ct=(\cn[<\hspace{-3pt}0]^{\perp_{\ct}})*(\cn[\ge\hspace{-3pt}0]^{\perp_{\ct}})$ holds.

Applying Lemma~\ref{no Homs} to the silting subcategory $\cn$ of $\ct$ we obtain 
\begin{equation}\label{describe N[<0]perp 1}
\cn[<\hspace{-3pt}0]^{\perp_{\ct}}=\bigcup_{i\ge0}\cn*\cn[1]*\cdots*\cn[i].
\end{equation}
Since $\cm$ and $\cn$ are compatible, we may assume, up to shift, that 
\begin{eqnarray}\label{eq:position-of-N}
\cn&\subset&\cm*\cm[1]*\cdots*\cm[n],\\
\label{M in terms of N}
\cm&\subset&\cn[-n]*\cdots*\cn[-1]*\cn,
\end{eqnarray}
for some integer $n$.
With (\ref{describe N[<0]perp 1}), (\ref{eq:position-of-N}) and (\ref{M in terms of N}), one can easily check that
\begin{equation}\label{describe N[<0]perp 2}
\cn[<\hspace{-3pt}0]^{\perp_{\ct}}=\bigcup_{i\ge n}\cn*\cdots*\cn[n-1]*\cm[n]*\cdots*\cm[i]
\end{equation}
holds. Now fix an integer $\ell\geq 2n-2$ and define subcategories $\cx$ and $\cy$ of $\ct$ by
\[\cx:=\cn*\cn[1]*\cdots*\cn[\ell]
\ \mbox{ and }\ \cy:=\cx^{\perp_\ct}.\]

Since $\cx\subset\cn[<\hspace{-3pt}0]^{\perp_{\ct}}$, it follows from Lemma~\ref{decompose Y} below that
\[\ct=\cx*\cy\subset(\cn[<\hspace{-3pt}0]^{\perp_{\ct}})*(\cn[<\hspace{-3pt}0]^{\perp_{\ct}})*(\cn[\ge\hspace{-3pt}0]^{\perp_{\ct}})=(\cn[<\hspace{-3pt}0]^{\perp_{\ct}})*(\cn[\ge\hspace{-3pt}0]^{\perp_{\ct}}).\]
This completes the proof.
\end{proof}

\begin{lemma}\label{decompose Y} Let $\ell$ be a non-negative integer and define subcategories $\cx$ and $\cy$ of $\ct$ by
\[\cx:=\cn*\cn[1]*\cdots*\cn[\ell]
\ \mbox{ and }\ \cy:=\cx^{\perp_\ct}.\]
\begin{itemize}
\item[(a)] We have $\ct=\cx*\cy$.
\item[(b)] If $\ell\ge 2n-2$, then $\cy\subset(\cn[<\hspace{-3pt}0]^{\perp_{\ct}})*(\cn[\ge\hspace{-3pt}0]^{\perp_{\ct}})$.
\end{itemize}
\end{lemma}
In the case when $\ct$ is Krull--Schmidt, part (a) is \cite[Proposition 2.4]{IyamaYoshino08}.
\begin{proof}
(a) Fix any $T_0\in\ct$. Since $\cn$ is a contravariantly finite subcategory of $\ct$, there exists a triangle
\[\xymatrix{N_i[i]\ar[r]^{f_i}&T_i\ar[r]&T_{i+1}\ar[r]&N_i[i+1]}\]
for each $0\le i\le\ell$ with a right $\cn[i]$-approximation $f_i$ of $T_i$.
Inductively, one can check that $\Hom_{\ct}(\cn[j],T_i)=0$ holds for any $0\le j<i$.
In particular, $T_{\ell+1}\in\cy$ holds.
Now, using $T_i\in\cn[i]*T_{i+1}$ repeatedly, we have
\[T_0\in\cn*T_1\subset\cn*\cn[1]*T_2\subset\cdots\subset\cn*\cn[1]*\cdots*\cn[\ell]*T_{\ell+1}\subset\cx*\cy\]
as desired.

(b) For any $Y\in\cy$, we take the triangle associated to the t-structure $(\ct^{\le-\ell-1},\ct^{\ge-\ell-1}):=(\cm[<\hspace{-3pt}0]^{\perp_{\ct}}[\ell+1],\cm[>\hspace{-3pt}0]^{\perp_{\ct}}[\ell+1])$
\begin{equation}\label{1-2d and 2-2d}
\xymatrix{\sigma^{\leq-\ell-1}Y\ar[r]&Y\ar[r]&\sigma^{\geq -\ell}Y\ar[r]&(\sigma^{\leq-\ell-1}Y)[1].}
\end{equation}
It suffices to show $\sigma^{\leq-\ell-1}Y\in \cn[<\hspace{-3pt}0]^{\perp_{\ct}}$ and $\sigma^{\geq-\ell}Y\in \cn[\ge\hspace{-3pt}0]^{\perp_{\ct}}$.

We have that $\sigma^{\leq-\ell-1}Y$ belongs to $\ct^{\leq-\ell-1}$, which, by \eqref{describe N[<0]perp 2}, is a subcategory of $\cn[<\hspace{-3pt}0]^{\perp_{\ct}}$. The first assertion follows.

To prove the second assertion, we need to show  $\Hom_\ct(\cn[<\hspace{-3pt}0]^{\perp_{\ct}},\sigma^{\geq -\ell}Y)=0$. By \eqref{describe N[<0]perp 2}, it suffices to show the following 
since $n-1\le\ell-n+1$: 
\begin{itemize}
\item[(i)] $\Hom_{\ct}(\cm[i],\sigma^{\geq-\ell}Y)=0$ for any $i$ with $\ell+1\le i$;
\item[(ii)] $\Hom_{\ct}(\cm[i],\sigma^{\geq-\ell}Y)=0$ for any $i$ with $n\le i\le\ell$;
\item[(iii)] $\Hom_{\ct}(\cn[i],\sigma^{\geq-\ell}Y)=0$ for any $i$ with $0\le i\le\ell-n+1$.
\end{itemize}

The statement (i) holds since $\sigma^{\geq-\ell}Y\in\ct^{\geq-\ell}$.

We show (ii). Since $(\sigma^{\leq-\ell-1}Y)[1]\in\ct^{\leq-\ell-2}$, we have $\Hom_\ct(\cm[i],(\sigma^{\leq-\ell-1}Y)[1])=0$ for any $i\le\ell+1$.
Since $Y\in\cy$ and $\cm[i]\in\cx=\cn*\cn[1]*\cdots*\cn[\ell]$ for any $n\le i\le\ell$ by \eqref{M in terms of N}, we have $\Hom_{\ct}(\cm[i],Y)=0$ for any $n\le i\le\ell$.
Thus the statement follows from the triangle \eqref{1-2d and 2-2d}.

We show (iii). Since $Y\in\cy$, we have $\Hom_{\ct}(\cn[i],Y)=0$ for any $0\le i\le\ell$.
Since $(\sigma^{\leq-\ell-1}Y)[1]\in\ct^{\leq-\ell-2}=\ct_{\geq-\ell-1}{}^{\perp_\ct}$ and $\cn\subset\ct_{\geq -n}$, we have $\Hom_{\ct}(\cn[i],(\sigma^{\leq-\ell-1}Y)[1])=0$ for any $0\le i\le \ell-n+1$. The statement follows from the triangle \eqref{1-2d and 2-2d}.
\end{proof}

\subsection{Hearts of adjacent t-structures}

In this subsection, we describe the heart of a t-structure right adjacent to a silting subcategory. 
We first prepare some notions.
For an additive category $\cm$, an \emph{$\cm$-module} is a contravariant additive functor from $\cm$ to the category of abelian groups.
We say that an $\cm$-module $F$ is \emph{finitely presented} if there
exist a sequence of natural transformations
\[\xymatrix{\Hom_{\cm}(-,M')\ar[r]&\Hom_{\cm}(-,M)\ar[r]& F\ar[r]&0}\]
with $M,M'\in\cm$ which is objectwise exact. We denote by $\mod\cm$ the category of 
finitely presented $\cm$-module.
Although $\mod\cm$ is in general not an abelian category, we have the following sufficient condition.

\begin{lemma}
Let $\ct$ be a triangulated category and $\cm$ a contravariantly (respectively, covariantly) finite subcategory of $\ct$. Then $\mod\cm$ (respectively, $\mod\cm^{\op}$) forms an abelian category.
\end{lemma}

\begin{proof}
Our assumptions imply that any morphism $f:M\to N$ in $\cm$ has a pseudokernel, that is, a morphism $g:M'\to M$ such that the sequence $\Hom_\cm(-,M')\stackrel{g\cdot}{\longrightarrow}\Hom_\cm(-,M)\stackrel{f\cdot}{\longrightarrow}\Hom_\cm(-,N)$ is exact.
Thus the assertion follows from the general result \cite[Chapter III, Section 2, the second Proposition]{Auslander71}.
\end{proof}

Now we have the following description of the heart of a t-structure right adjacent to a silting subcategory (compare: \cite[Theorem 1.3(c)]{HKM}, \cite[Chapter IV, Theorem 3.4]{BeligiannisReiten07} and \cite[Corollary 4.7]{PsaroudakisVitoria14}).

\begin{proposition}\label{l:heart-of-the-adjacent-t-structure}
Let $\ct$ be a triangulated category. 
\begin{itemize}
\item[(a)] If $\cm$ is a silting subcategory of $\ct$ and admits a right adjacent t-strucutre $(\cm[<\hspace{-3pt}0]^{\perp_{\ct}},\cm[>\hspace{-3pt}0]^{\perp_{\ct}})$, then the functor  $\Hom_{\ct}(\cm,-)\colon\ct\rightarrow\mod\cm$ restricts to an equivalence from the heart $\ch$ to $\mod\cm$.
\item[(b)] If $\cm$ is a silting subcategory of $\ct$ and admits a left adjacent t-structure $({}^{\perp_{\ct}}\cm[<\hspace{-3pt}0],{}^{\perp_{\ct}}\cm[>\hspace{-3pt}0])$, then the functor  $\Hom_{\ct}(-,\cm)\colon\ct\rightarrow\mod\cm^{\op}$ restricts to an anti-equivalence from the heart $\ch$ to $\mod\cm^{\op}$.
\end{itemize}
\end{proposition}

\begin{proof}
We only prove (a) since (b) follows dually. Let $(\ct^{\le0},\ct^{\ge0}):=(\cm[<\hspace{-3pt}0]^{\perp_{\ct}},\cm[>\hspace{-3pt}0]^{\perp_{\ct}})$. 
For any $M\in\cm$, consider the triangle 
\[
\xymatrix{
M^{\leq -1}\ar[r] & M\ar[r] & M^0 \ar[r] & M^{\leq -1}[1] & \text{($M^{\leq -1}\in \ct^{\leq -1}$ and $M^0\in\ch$).}}
\]
Let $\cm^0:=\{M^0\mid M\in\cm\}$. Then a direct diagram chase shows that the functor $(-)^0\colon\cm\to\cm^0$ is an equivalence. 
We have $\Hom(M^{\leq -1},\ch)=0$, and hence we have a commutative diagram
\[\xymatrix{
\ch\ar[d]_{\Hom_{\ct}(\cm^0,-)}\ar[drr]^{\Hom_{\ct}(\cm,-)}\\
\mod\cm^0\ar[rr]_{(-)^0}^{\simeq}&&\mod\cm.}\]
So, by Morita's theorem, it suffices to show that objects of $\cm^0$ form a class of projective generators of $\ch$.

Let $M\in\cm$. For any $X\in\ch$, applying $\Hom_\ct(-,X)$ to the triangle associated to $M$ as in the beginning of the proof, we obtain an exact sequence 
\[
\xymatrix@C=1pc{0=\Hom_\ct(M^{\leq -1},X)\ar[r] & \Hom_\ct(M^0,X[1])\ar[r] & \Hom_\ct(M,X[1])=0.}
\]
Thus $\Ext^1_\ch(M^0,X)\simeq\Hom_\ct(M^0,X[1])=0$. This shows that $M^0$ is projective in $\ch$, so objects of $\cm^0$ are projective in $\ch$.

For $X\in\ch$, take a right $\cm$-approximation $M^X\rightarrow X$ and form a triangle
\[
\xymatrix{
N^X\ar[r] & M^X\ar[r] & X\ar[r] & N^X[1].
}
\]
Applying $\Hom_\ct(\cm,-)$ to this triangle, we obtain long exact sequences
\[
\xymatrix@C=0.6pc{
\Hom_\ct(\cm,M^X[i-1])\ar[r] & \Hom_\ct(\cm,X[i-1])\ar[r] & \Hom_\ct(\cm,N^X[i])\ar[r] & \Hom_\ct(\cm,M^X[i]).
}
\]
We claim that $\Hom_\ct(\cm,N^X[i])=0$ for $i\geq 1$, hence $N^X\in\ct^{\leq 0}$. Indeed, $\Hom_\ct(\cm,M^X[i])$ vanishes for all $i\geq 1$. If $i=1$, then the left morphism is surjective; if $i>1$, then $\Hom_\ct(\cm,X[i-1])=0$. The claim follows immediately.  Now taking the 0th cohomology associated to the t-structure $(\ct^{\leq 0},\ct^{\geq 0})$, we obtain an exact sequence in $\ch$
\[
\xymatrix@R=0.5pc{
H^0(M^X)\ar[r]\ar@{=}[d] & H^0(X)\ar[r]\ar@{=}[d] & H^0(N^X[1]),\ar@{=}[d]\\
(M^X)^0& X & 0
}
\]
showing that $\cm^0$ consists of a class of projective generators of $\ch$.
\end{proof}

\subsection{Right and left adjacent t-structures}

In this subsection, under certain assumptions, we show that a silting subcategory has a right adjacent t-structure if and only if it has a left adjacent t-structures.

Let $k$ be a field, and let $D=\Hom_k(-,k)$ denote the $k$-dual.
We consider the following conditions:
\begin{itemize}
\item[(RS1)] $\ct$ is a $k$-linear Hom-finite triangulated category and $\ct^{\fd}$ is a thick subcategory of $\ct$.
\item[(RS2)] $\ct^{\fd}$ has an auto-equivalence $S$ such that a \emph{relative Serre duality} holds in the sense that there exists a functorial isomorphism for any $X\in\ct^{\fd}$ and $Y\in\ct$
$$D\Hom_{\ct}(X,Y)\simeq \Hom_{\ct}(Y,SX).$$
\end{itemize}
In this case, we prove the following.

\begin{theorem}\label{left-right 1} Under the assumptions (RS1) and (RS2), let $M$ be a silting object of $\ct$.
The following conditions are equivalent.
\begin{itemize}
\item[(a)] $M$ has a right adjacent t-structure $(M[<\hspace{-3pt}0]^{\perp_{\ct}},M[>\hspace{-3pt}0]^{\perp_{\ct}})$ with $M[>\hspace{-3pt}0]^{\perp_{\ct}}\subset\ct^{\fd}$.
\item[(b)] $M$ has a left adjacent t-structure
$({}^{\perp_{\ct}} M[<\hspace{-3pt}0],{}^{\perp_{\ct}} M[>\hspace{-3pt}0])$ with ${}^{\perp_{\ct}} M[<\hspace{-3pt}0]\subset\ct^{\fd}$.
\end{itemize}
In this case, we have $S({}^{\perp_{\ct}} M[<\hspace{-3pt}0])\subset M[<\hspace{-3pt}0]^{\perp_{\ct}}$ and ${}^{\perp_{\ct}} M[>\hspace{-3pt}0]\supset S^{-1}(M[>\hspace{-3pt}0]^{\perp_{\ct}})$, and $S$ restricts to an equivalence $S\colon{}^{\perp_{\ct}} M[<\hspace{-3pt}0]\cap {}^{\perp_{\ct}} M[>\hspace{-3pt}0]\to M[<\hspace{-3pt}0]^{\perp_{\ct}}\cap M[>\hspace{-3pt}0]^{\perp_{\ct}}$ of hearts.
\end{theorem}

In fact we will prove a more general result for silting subcategories.
Let $\cm$ be a $k$-linear Hom-finite additive category. Then any $\cm$-module $F$ can be naturally regarded as a contravariant $k$-linear functor $\cm\to\Mod k$. We define an $\cm^{\op}$-module $DF$ as the composition
\[DF:=(\cm\xrightarrow{F}\Mod k\xrightarrow{D}\Mod k).\]
We say that $\cm$ is a \emph{dualizing $k$-variety} \cite{AuslanderReiten74} if the following conditions are satisfied.
\begin{itemize}
\item For any $F\in\mod\cm$, the functor $DF$ belongs to $\mod\cm^{\op}$.
\item For any $F\in\mod\cm^{\op}$, the functor $DF$ belongs to $\mod\cm$.
\end{itemize}
In this case, we have anti-equivalences $D:\mod\cm\leftrightarrow\mod\cm^{\op}$, and $\mod\cm$ has enough projective objects $\proj\cm$ and injective objects $\inj\cm$.
We have an equivalence
\[\nu:\proj\cm\xrightarrow{\simeq}\inj\cm\ \mbox{ given by }\ \nu(\Hom_{\cm}(-,M)):=D\Hom_{\cm}(M,-)\]
for $M\in\cm$, which we call the \emph{Nakayama functor}.

Since any $k$-linear Hom-finite category which has an additive generator is a dualizing $k$-variety, Theorem \ref{left-right 1} follows from the following result.

\begin{theorem}\label{left-right} Under the assumptions (RS1) and (RS2), let $\cm$ be a silting subcategory of $\ct$ and assume that $\cm$ is a dualizing $k$-variety.
Then the following conditions are equivalent.
\begin{itemize}
\item[(a)] $\cm$ has a right adjacent t-structure $(\cm[<\hspace{-3pt}0]^{\perp_{\ct}},\cm[>\hspace{-3pt}0]^{\perp_{\ct}})$ with $\cm[>\hspace{-3pt}0]^{\perp_{\ct}}\subset\ct^{\fd}$.
\item[(b)] $\cm$ has a left adjacent t-structure
$({}^{\perp_{\ct}} \cm[<\hspace{-3pt}0],{}^{\perp_{\ct}} \cm[>\hspace{-3pt}0])$ with ${}^{\perp_{\ct}} \cm[<\hspace{-3pt}0]\subset\ct^{\fd}$.
\end{itemize}
In this case, we have $S({}^{\perp_{\ct}} \cm[<\hspace{-3pt}0])\subset\cm[<\hspace{-3pt}0]^{\perp_{\ct}}$ and
${}^{\perp_{\ct}} \cm[>\hspace{-3pt}0]\supset S^{-1}(\cm[>\hspace{-3pt}0]^{\perp_{\ct}})$,  and $S$ restricts to an equivalence $S\colon{}^{\perp_{\ct}} \cm[<\hspace{-3pt}0]\cap {}^{\perp_{\ct}} \cm[>\hspace{-3pt}0]\to\cm[<\hspace{-3pt}0]^{\perp_{\ct}}\cap\cm[>\hspace{-3pt}0]^{\perp_{\ct}}$ of hearts; moreover, $\cm$ is a functorially finite subcategory of $\ct$.
\end{theorem}

Before proving Theorem~\ref{left-right}, we give the following characterization of the subcategory $\ct^{\fd}$ of $\ct$, which justifies the notation.

\begin{lemma}\label{l:fd-has-fd-cohomology}Let $\cm$ be a silting subcategory of $\ct$ and let $X$ be an object of $\ct$. Consider the following conditions:
\begin{itemize}
\item[(a)] $X$ belongs to $\ct^{\fd}$;
\item[(b)] the space $\Hom_{\ct}(\cm,X[i])$ vanishes for almost all $i\in\Z$;
\item[(c)] the space $\Hom_{\ct}(X[i],\cm)$ vanishes for almost all $i\in\mathbb{Z}$.
\end{itemize}
Then {\rm(a)} implies {\rm(b)} and {\rm(c)}. If $\cm[>\hspace{-3pt}0]^{\perp_{\ct}}\subset\ct^{\fd}$, then {\rm(a)} and {\rm(b)} are equivalent; if ${}^{\perp_{\ct}} \cm[<\hspace{-3pt}0]\subset\ct^{\fd}$, then {\rm(a)} and {\rm(c)} are equivalent.
\end{lemma}

\begin{proof}
(a)$\Rightarrow$(b): Let $X\in\ct^{\fd}$. Then we have $\Hom_\ct(\cm,X[i])=0$ and $\Hom_\ct(\cm,X[-i])\simeq D\Hom_\ct(S^{-1}X,\cm[i])=0$ for $i\gg0$ by Lemma \ref{vanishing for >>}.

(b)$\Rightarrow$(a): Assume $\cm[>\hspace{-3pt}0]^{\perp_{\ct}}\subset\ct^{\fd}$. If (b) holds, then there exists an integer $i$ such that $\Hom_\ct(\cm,X[<\hspace{-3pt}i])=0$, \emph{i.e.} $X\in\cm[>\hspace{-3pt}-i]^{\perp_\ct}$. Since $\cm[>\hspace{-3pt}-i]^{\perp_\ct}=\cm[>\hspace{-3pt}0]^{\perp_\ct}[-i]$  is contained in $\ct^{\fd}$, it follows that $X$ belongs to $\ct^{\fd}$ . 

(a)$\Rightarrow$(c) and (c)$\Rightarrow$(a): Dual to (a)$\Rightarrow$(b) and (b)$\Rightarrow$(a) respectively.
\end{proof}

The ``moreover" part of Theorem~\ref{left-right} is a consequence of Proposition~\ref{p:adjacent-->finiteness}.
In the rest of this section, we prove that (a) implies (b). Then the converse follows by Lemma \ref{basic right and left}.
Let $(\ct_{\ge0},\ct_{\le0})$ be the co-t-structure associated to $\cm$.
We denote by $\ch$ the heart of the t-structure $(\ct^{\le0},\ct^{\ge0}):=(\cm[<\hspace{-3pt}0]^{\perp_{\ct}},\cm[>\hspace{-3pt}0]^{\perp_{\ct}})$.
We denote by $\sigma^{\le i}$ and $\sigma^{\ge i+1}$ the  truncation functors associated with the t-structures $(\ct^{\leq i},\ct^{\geq i}):=(\ct^{\leq 0}[-i],\ct^{\geq 0}[-i])$. Then, for any $X\in\ct^{\le0}=\ct_{\le0}$, there exists a triangle
\[\xymatrix{L[-1]\ar[r]& Y\ar[r]& X\ar[r]& L}\]
in $\ct$ with $L=\sigma^{\ge0}X\in\ch$ and $Y=\sigma^{\le-1}X\in\ct^{\le-1}=\ct_{\le-1}$.

The following dual statement is a crucial step to prove that $({}^{\perp_{\ct}} \cm[<\hspace{-3pt}0],{}^{\perp_{\ct}} \cm[>\hspace{-3pt}0])$ forms a t-structure on $\ct$.
It was inspired by the result \cite[Lemma 2.9]{Guolingyan11a} of Guo.

\begin{proposition}\label{one step}
For any $X\in\ct_{\ge0}$, there exists a triangle
\[\xymatrix{S^{-1}(L)\ar[r]& X\ar[r]& Y\ar[r]& S^{-1}(L)[1]}\]
in $\ct$ with $L\in\ch$ and $Y\in\ct_{\ge1}$.
In particular, we have $\ct_{\ge0}=S^{-1}(\ch)*\ct_{\ge1}$.
\end{proposition}

\begin{proof}
It suffices to prove the first assertion. In fact, it implies $\ct_{\ge0}\subset S^{-1}(\ch)*\ct_{\ge1}$.
Then the equality holds since $S^{-1}(\ch)\subset{}^{\perp_{\ct}}\cm[>\hspace{-3pt}0]=\ct_{\ge0}$ holds by the relative Serre duality.

Fix $X\in\ct_{\ge0}$.  Take a triangle
\begin{equation}\label{co-t-structure for X}
\xymatrix{X_{\ge 2}\ar[r]&X\ar[r]&W[-1]\ar[r]&X_{\ge 2}[1]}
\end{equation}
with $X_{\ge 2}\in\ct_{\ge 2}$ and $W\in\cm*\cm[1]$.
Then there exists a triangle
\begin{equation}\label{co-t-structure for X2}
\xymatrix{M_1\ar[r]^f&M_0\ar[r]&W\ar[r]&M_1[1]}
\end{equation}
with $M_0,M_1\in\cm$. By Proposition \ref{l:heart-of-the-adjacent-t-structure}, the functor $F:=\Hom_\ct(\cm,-)\colon\ct\to\mod\cm$ induces an equivalence
\[F\colon\ch\xrightarrow{\simeq}\mod\cm.\]
Since $\cm$ is a dualizing $k$-variety by our assumption, we have the Nakayama functor $\nu\colon\proj\cm\xrightarrow{\simeq}\inj\cm$.
We define $L\in\ch$ by the exact sequence in $\mod\cm$:
\begin{equation}\label{define L}
\xymatrix{0\ar[r]&F(L)\ar[r]&\nu F(M_1)\ar[r]^{\nu F(f)}&\nu F(M_0)}.
\end{equation}
(This means that $F(L)$ is the Auslander--Reiten translation of $F(W)$ unless $W$ has direct summands in $\cm[1]$.) To continue the proof we need the following lemma.

\begin{lemma}\label{define g}
There exists a morphism $g\in\Hom_\ct(S^{-1}(L),X)$ which induces a functorial isomorphism for $U\in\ct^{\le0}$:
\[\Hom_\ct(g,U)\colon\Hom_\ct(X,U)\xrightarrow{\simeq}\Hom_\ct(S^{-1}(L),U).\]
\end{lemma}

\begin{proof}
We first show that there are the following functorial isomorphisms:
\begin{itemize}
\item[(i)] $\Hom_\ct(X,U)\simeq\Hom_\ct(W[-1],U)$;
\item[(ii)] $\Hom_\ct(W[-1],U)\simeq D\Hom_\cm(F(U),F(L))$;
\item[(iii)] $D\Hom_\cm(F(U),F(L))\simeq\Hom_\ct(S^{-1}(L),U)$.
\end{itemize}

By the triangle \eqref{co-t-structure for X}, we have an exact sequence
\[
\xymatrix@C=0.7pc{\Hom_\ct(X_{\ge2}[1],-)\ar[r]&\Hom_\ct(W[-1],-)\ar[r]&\Hom_\ct(X,-)\ar[r]&\Hom_\ct(X_{\ge2},-).}
\]
Evaluated at $U$, this gives the functorial isomorphism (i), since  $\Hom_\ct(X_{\ge 2}[\le\hspace{-3pt} 1],U)=0$.

The triangle \eqref{co-t-structure for X2} and the exact sequence \eqref{define L} yield a  commutative diagram with exact rows:
\[
\xymatrix{
0\ar[r]&D\Hom_\ct(W,U[1])\ar[r]&
D\Hom_\ct(M_1,U)\ar[r]^{D(\cdot f)}\ar[d]&D\Hom_\ct(M_0,U)\ar[d]\\
0\ar[r]&\Hom_\cm(F(U),F(L))\ar[r]&\Hom_\cm(F(U),\nu F(M_1))\ar[r]^{\nu F(f)\cdot}&\Hom_\cm(F(U),\nu F(M_0)).
}
\]
Here we used the vanishing of $D\Hom_\ct(M_0,U[1])$. The vertical arrows are the functorial isomorphism for $M\in\cm$
\[
\Hom_\ct(M,U)\simeq\Hom_\cm(F(M),F(U))\simeq D\Hom_\cm(F(U),\nu F(M)).
\]
As a consequence, the diagram gives us the functorial isomorphism (ii).

Since $\sigma^{\ge0}U\in\ch$ and $F(U)\simeq F(\sigma^{\ge0}U)$, we have
functorial isomorphisms
\[\Hom_\cm(F(U),F(L))\simeq\Hom_\cm(F(\sigma^{\ge0}U),F(L))
\simeq\Hom_\ct(\sigma^{\ge0}U,L)\simeq\Hom_\ct(U,L).\]
Using the relative Serre duality, we obtain the functorial isomorphism (iii).

Composing (i), (ii) and (iii), we have a functorial isomorphism $$\Hom_\ct(X,U)\simeq\Hom_\ct(S^{-1}(L),U)$$ for $U\in\ct^{\le0}$.
Using the relative Serre duality, we have a functorial isomorphism $$\Hom_{\ct}(-,S^{-1}(L))\simeq\Hom_\ct(-,X)$$ on $S^{-1}(\ct^{\le0}\cap\ct^{\fd})$. 
This is induced by a morphism $g\in\Hom_\ct(S^{-1}(L),X)$
by Yoneda's Lemma since $S^{-1}(L)$ belongs to $S^{-1}(\ct^{\le0}\cap\ct^{\fd})$.
\end{proof}

Now we continue the proof of Proposition~\ref{one step}. We extend the morphism $g$ given in Lemma \ref{define g} to a triangle
\begin{equation}\label{construct Y}
\xymatrix{Y[-1]\ar[r]&S^{-1}(L)\ar[r]^(0.6)g&X\ar[r]&Y.}
\end{equation}
It suffices to prove $Y\in\ct_{\ge1}$, that is, $\Hom_{\ct}(Y,\cm[\ge\hspace{-3pt}0])=0$.
Since
$\Hom_\ct(X,\cm[\ge\hspace{-3pt} 1])=0$ and
$\Hom_\ct(S^{-1}(L),\cm[\neq\hspace{-3pt} 0])=D\Hom_\ct(\cm,L[\neq\hspace{-3pt}0])=0$ by $L\in\ch$,
it follows that $\Hom_\ct(Y,\cm[\ge\hspace{-3pt}2])=0$. Moreover, we have an exact sequence
\begin{eqnarray*}
0=\Hom_\ct(S^{-1}(L),\cm[-1])&\to&\Hom_\ct(Y,\cm)\to
\Hom_\ct(X,\cm)\xrightarrow{\cdot g}\Hom_\ct(S^{-1}(L),\cm)\\
&\to&\Hom_\ct(Y,\cm[1])\to\Hom_\ct(X,\cm[1])=0.
\end{eqnarray*}
By Lemma \ref{define g}, the map $g$ is bijective, and hence $\Hom_\ct(Y,\cm)=0=\Hom_\ct(Y,\cm[1])$.
So $Y\in\ct_{\ge 1}$ and the proof is complete.
\end{proof}

Now we are ready to prove Theorem~\ref{left-right}. 

\begin{proof}[Proof of Theorem~\ref{left-right}.]
We only show that (a) implies (b). The converse follows by Lemma \ref{basic right and left}.

Since ${}^{\perp_{\ct}} \cm[\le\hspace{-3pt}0]={}^{\perp_{\ct}}\ct_{\ge0}$ and ${}^{\perp_{\ct}} \cm[>\hspace{-3pt}0]=\ct_{\ge0}$ hold,
$\Hom_{\ct}({}^{\perp_{\ct}} \cm[\le\hspace{-3pt}0],{}^{\perp_{\ct}} \cm[>\hspace{-3pt}0])=0$ holds.
To prove that $({}^{\perp_{\ct}} \cm[<\hspace{-3pt}0],{}^{\perp_{\ct}} \cm[>\hspace{-3pt}0])$ is a t-structure, it is enought to show $\ct=({}^{\perp_{\ct}}\ct_{\ge0})*\ct_{\ge0}$.
Since $\ct=\bigcup_{\ell\ge0}\ct_{\ge-\ell}$, it is enough to show $\ct_{\ge-\ell}\subset({}^{\perp_{\ct}}\ct_{\ge0})*\ct_{\ge0}$.
Using Proposition \ref{one step} repeatedly, we have
\[\ct_{\ge-\ell}\subset S^{-1}(\ch[\ell])*\ct_{\ge1-\ell}\subset S^{-1}(\ch[\ell])*S^{-1}(\ch)[\ell-1]*\ct_{\ge2-\ell}\subset\cdots\]
and hence
\begin{equation}\label{heart extension}
\ct_{\ge-\ell}\subset S^{-1}(\ch)[\ell]*S^{-1}(\ch)[\ell-1]*\cdots*S^{-1}(\ch)[1]*\ct_{\ge0}.
\end{equation}
This shows the desired equality $({}^{\perp_{\ct}}\ct_{\ge0})*\ct_{\ge0}=\ct$
since by the relative Serre duality $S^{-1}(\ch)[\ell]*\cdots*S^{-1}(\ch)[1]\subseteq{}^{\perp_{\ct}}\ct_{\ge0}$ holds.
Thus $({}^{\perp_{\ct}} \cm[<\hspace{-3pt}0],{}^{\perp_{\ct}} \cm[>\hspace{-3pt}0])$ is a t-structure.

Now we show ${}^{\perp_{\ct}}\ct_{\ge0}\subset\ct^{\fd}$.
For any $X\in{}^{\perp_{\ct}}\ct_{\ge0}$, we take $\ell\gg0$ such that $X\in\ct_{\ge-\ell}$. Applying Lemma \ref{belong to add} to \eqref{heart extension}, we have $X\in\thick S^{-1}(\ch)\subset\ct^{\fd}$.

The remaining statements
follow immediately from the relative Serre duality.
\end{proof}

\section{Silting reduction versus Calabi--Yau reduction}
\label{s:silting-red-vs-cy-red}

In Theorems~\ref{equivalence} and \ref{t:triangle-equivalence}, we realise silting reduction as subfactor categories. This is analogous to the Calabi--Yau reduction introduced by Yoshino and the first author in \cite{IyamaYoshino08}. In this section, we relate these two constructions, using the results in the preceding sections.
We will show that silting reduction of Calabi--Yau triangulated categories induces Calabi--Yau reduction (Theorem~\ref{main theorem new}).

Throughout this section, let $k$ be a field and let $D=\Hom_k(-,k)$ denote the $k$-dual.  Let $d\geq 1$ be an integer.

\subsection{Calabi--Yau triples}\label{ss:setup}
Let $\ct$ be $k$-linear  triangulated category, $\cm$ a subcategory of $\ct$ and $\ct^{\fd}$ a triangulated subcategory of $\ct$. 
We say that $(\ct,\ct^{\fd},\cm)$ is a \emph{$(d+1)$-Calabi--Yau triple} if the following conditions are satisfied.
\begin{itemize}
\item[(CY1)] The category $\ct$ is Hom-finite and Krull--Schmidt.
\item[(CY2)] The pair $(\ct,\ct^{\fd})$ is \emph{relative $(d+1)$-Calabi--Yau}
in the sense that there exists a bifunctorial isomorphism for any $X\in\ct^{\fd}$ and $Y\in\ct$:
\[D\Hom_{\ct}(X,Y)\simeq \Hom_{\ct}(Y,X[d+1]).\]
\item[(CY3)] The subcategory $\cm$ is a silting subcategory of $\ct$ and admits a right adjacent t-structure $(\ct^{\le0},\ct^{\ge0}):=(\cm[<\hspace{-3pt}0]^{\perp_\ct},\cm[>\hspace{-3pt}0]^{\perp_\ct})$ with $\ct^{\geq 0}\subset\ct^{\fd}$. Moreover, $\cm$ is a dualizing $k$-variety.
\end{itemize}

It follows from Theorem~\ref{left-right} that $\cm$ is a functorially finite subcategory of $\ct$. We remark that the condition that $\cm$ is a dualizing $k$-variety will not be used in this section and Section~\ref{ss:silting-reduction-is-3cy} but will be crucial in Sections~\ref{ss:AGK-cluster-category} and \ref{ss:silting-red-vs-cy-red}. We remind the reader that if $\cm=\add M$ is the additive closure of a silting object $M$ then $\cm$ is automatically a dualizing $k$-variety. By Theorem~\ref{left-right} again,  (CY3) is equivalent to its dual:
\begin{itemize}
\item[(CY3$^{\op}$)] The subcategory $\cm$ is a silting subcategory of $\ct$ and admits a left adjacent t-structure $({}^{\perp_{\ct}}\cm[<\hspace{-3pt}0],{}^{\perp_{\ct}}\cm[>\hspace{-3pt}0])$ with ${}^{\perp_{\ct}}\cm[<\hspace{-3pt}0]\subset\ct^{\fd}$. Moreover, $\cm$ is a dualizing $k$-variety.
\end{itemize}

Note that the condition (CY3) is independent of the choice of $\cm$ in the following sense:

\begin{remark}\label{idependence of M}
Let $\cm$ and $\cn$ be silting subcategories of $\ct$ which are dualizing $k$-varieties and compatible with each other.  Then $(\ct,\ct^{\fd},\cm)$ is a $(d+1)$-Calabi--Yau triple if and only if $(\ct,\ct^{\fd},\cn)$ is a $(d+1)$-Calabi--Yau triple.
\end{remark}

\begin{proof}
We will show the `only if' part. By Theorem \ref{l:shift-t-str}, $\cn$ admits a right adjacent t-structure $(\cn[<\hspace{-3pt}0]^{\perp_\ct},\cn[>\hspace{-3pt}0]^{\perp_\ct})$.
 Take $\ell\gg0$ such that $\cm\subset\cn[-\ell]*\cn[1-\ell]*\cdots*\cn[\ell-1]*\cn[\ell]$.
Then $\cn[>\hspace{-3pt}0]^{\perp_\ct}\subset \cm[>\hspace{-3pt}\ell]^{\perp_\ct}\subset \ct^{\fd}$. Thus
 $(\ct,\ct^\fd,\cn)$ is a $(d+1)$-Calabi--Yau triple.
\end{proof}

In the rest of this subsection, let $(\ct,\ct^\fd,\cm)$ be a $(d+1)$-Calabi--Yau triple. For simplicity, we assume $\cm=\add\cm$. Put
\begin{eqnarray*}
\ct_{\leq 0}&:=&\bigcup_{i\ge0}\cm*\cm[1]*\cdots*\cm[i],\\
\ct_{\geq 0}&:=&\bigcup_{i\ge0}\cm[-i]*\cdots*\cm[-1]*\cm.
\end{eqnarray*}
Then $(\ct_{\geq 0},\ct_{\leq 0})$ is a bounded co-t-structure on $\ct$ with co-heart $\cm$, by Proposition~\ref{from silting to co-t-structure}.
As a consequence, $\ct_{\leq 0}=\ct_{\geq 0}[-1]^{\perp_{\ct}}=\cm[<\hspace{-3pt}0]^{\perp_\ct}=\ct^{\leq 0}$. Moreover, since $\ct^{\fd}$ is closed under shifts, we have $\ct^{\geq i}\subset \ct^{\fd}$ for any $i\in\mathbb{Z}$.

Now we show that the t-structure $(\ct^{\leq 0},\ct^{\geq 0})$ restricts to a t-structure on $\ct^{\fd}$.

\begin{lemma}\label{l:t-str-restricts-to-fd-objects}
The pair $(\ct^{\fd}\cap \ct^{\leq 0},\ct^{\geq 0})$ is a bounded t-structure on $\ct^{\fd}$. It has the same heart $\ch$ as $(\ct^{\leq 0},\ct^{\geq 0})$. Consequently, $\ct^{\fd}$ is the smallest triangulated subcategory of $\ct$ containing $\ch$.
\end{lemma}
\begin{proof}
For $X\in\ct^{\fd}$, there is a triangle
\[
\xymatrix{
\sigma^{\leq 0}X\ar[r] & X\ar[r] & \sigma^{\geq 1}X\ar[r] & (\sigma^{\leq 0}X)[1].
}
\]
Since both $X$ and $\sigma^{\geq 1}X$ belong to the triangulated subcategory $\ct^{\fd}$ of $\ct$, it follows that $\sigma^{\leq 0}X$ belongs to $\ct^{\fd}$ and hence to $\ct^{\fd}\cap\ct^{\leq 0}$. This shows that $(\ct^{\fd}\cap \ct^{\leq 0},\ct^{\geq 0})$ is a t-structure on $\ct^{\fd}$.

Let $X$ be any object of $\ct^{\fd}$. By Lemma~\ref{l:fd-has-fd-cohomology}, there exist integers $i\leq j$ such that $\Hom_\ct(\cm,X[<\hspace{-3pt}i])=0$ and $\Hom_\ct(\cm,X[>\hspace{-3pt}j])=0$. Namely, $X$ belongs to $\ct^{\fd}\cap\ct^{\leq j}\cap\ct^{\geq i}$.  By definition the t-structure $(\ct^{\fd}\cap \ct^{\leq 0},\ct^{\geq 0})$ is bounded.

The second statement holds true because $\ct^{\geq 0}\subset \ct^{\fd}$.
\end{proof}

\begin{remark}
Assume further that $\ct$ is algebraic and $\cm=\add M$ is the additive closure of a silting object $M$. Then there is a dg algebra $A$ such that there is a triangle equivalence $\ct\to\per(A)$ which takes $M$ to $A$, see Section~\ref{ss:derived-category}.
It follows that $H^i(A)\simeq\Hom_{\per(A)}(A,A[i])\simeq\Hom_\ct(M,M[i])=0$ for $i>0$ and $H^0(A)\simeq\End_{\per(A)}(A)\simeq\End_\ct(M)$ is finite-dimensional over $k$.
Let 
\[\ch:=\{X\in\per(A)\mid H^i(X)=0\mbox{ for all }i\neq0\}.\]
By Proposition~\ref{l:heart-of-the-adjacent-t-structure}, we have an equivalence
\[H^0=\Hom_{\per(A)}(A,-)\colon\ch\to\mod H^0(A).\]
Therefore we have an equality
\[\ch=\{X\in\cd(A)\mid H^i(X)=0\mbox{ for any }i\neq0,\ H^0(X)\in\mod H^0(A)\},\]
which implies $\per(A)\supset\cd_{\fd}(A)$, since $\cd_{\fd}(A)$ is the smallest triangulated subcategory of $\cd(A)$ containing $\ch$, see for example \cite[Proposition 2.5(b)]{KalckYang12}. Comparing this with Lemma~\ref{l:t-str-restricts-to-fd-objects}, we obtain that the equivalence $\ct\to\per(A)$ restricts to a triangle equivalence $\ct^{\fd}\rightarrow \cd_{\fd}(A)$. Thus the dg algebra $A$ satisfies the following conditions
\begin{itemize}
\item[(1)] $H^i(A)=0$ for $i>0$;
\item[(2)] $H^0(A)$ is finite-dimensional over $k$;
\item[(3)] $\per(A)\supset \cd_{\fd}(A)$;
\item[(4)] there is a bifunctorial isomorphism for $X\in\cd_{\fd}(A)$ and $Y\in\per(A)$
\[
D\Hom_{\per(A)}(X,Y)\simeq \Hom_{\per(A)}(Y,X[d+1]).
\]
\end{itemize}
This is very close to the original setting of Amiot in \cite[Section 2]{Amiot09} and of Guo in \cite[Section 1]{Guolingyan11a}.
\end{remark}

\subsection{The silting reduction of a Calabi--Yau triple}\label{ss:silting-reduction-is-3cy}

Let $(\ct,\ct^{\fd},\cm)$ be a $(d+1)$-Calabi--Yau triple, as in Section~\ref{ss:setup}.  
Let $\cp$ be a functorially finite subcategory of $\cm$. Then $\cp$ is a presilting subcategory of $\ct$ satisfying the conditions (P1) and (P2) in Section~\ref{ss:add-equiv}. Let
\[
\SS:=\thick\cp,\ \ \ \cu:=\ct/\SS.
\] 
Let $\rho\colon\ct\to\cu$ be the canonical projection functor. By abuse of notation, we will write $\cm$ for $\rho(\cm)$.
By the relative $(d+1)$-Calabi--Yau property (CY2), we have $\ct^{\fd}\cap\SS^{\perp_\ct}=\ct^{\fd}\cap\, {}^{\perp_\ct}\SS$, which will be denoted by $\cu^{\fd}$, i.e.
\[
\cu^{\fd}:=\ct^{\fd}\cap\SS^{\perp_\ct}=\ct^{\fd}\cap\, {}^{\perp_\ct}\SS.
\]
This category can be considered as a full subcategory of $\cu$ (by, for example, \cite[Lemma 9.1.5]{Neeman01b}).

\begin{theorem}\label{p:silting-reduction-is-3cy} The triple $(\cu,\cu^{\fd},\cm)$ is a $(d+1)$-Calabi--Yau triple. Namely,
\begin{itemize}
\item[(a)] $\cu$ is Hom-finite and Krull--Schmidt.
\item[(b)] The pair $(\cu,\cu^{\fd})$ is relative $(d+1)$-Calabi--Yau.
\item[(c)] The subcategory $\cm$ of $\cu$ is a dualizing $k$-variety. It is a silting subcategory of $\cu$ and admits a right adjacent t-structure $(\cm[<\hspace{-3pt}0]^{\perp_{\cu}},\cm[>\hspace{-3pt}0]^{\perp_{\cu}})$ with $\cm[>\hspace{-3pt}0]^{\perp_{\cu}}\subset\cu^{\fd}$. 
\end{itemize}
\end{theorem}

In the proof of this theorem a crucial role is played by the following description of $\cu$ obtained in Section~\ref{s:subfactor-category}: Let
\begin{equation}\label{P and Z}
\cz:=({}^{\perp_{\ct}}\SS_{<0})\cap(\SS_{>0}{}^{\perp_{\ct}}),
\end{equation}
then we have a triangle equivalence (Theorems~\ref{equivalence} and \ref{t:triangle-equivalence})
\[
G\colon\frac{\cz}{[\cp]}\stackrel{\simeq}{\longrightarrow}\cu.
\]
Our strategy is to show that under $G$ the triple $(\cu,\cu^{\fd},\cm)$ is equivalent to $(\frac{\cz}{[\cp]},\ct^{\fd}\cap\cz,\frac{\cm}{[\cp]})$ and then prove Theorem~\ref{p:silting-reduction-is-3cy} for $(\frac{\cz}{[\cp]},\ct^{\fd}\cap\cz,\frac{\cm}{[\cp]})$.
We need some further preparation.

\begin{lemma}\label{l:fd-objects-of-Z}
We have an equality $\cu^{\fd}=\ct^{\fd}\cap \cz$ of subcategories of $\ct$.
\end{lemma}
\begin{proof}
Let $X\in\ct^{\fd}$. Then $X\in\cz$ if and only if $\Hom_\ct(X,\SS_{<0})=0$ and $\Hom_\ct(\SS_{>0},X)=0$. By the relative $(d+1)$-Calabi--Yau property, this amounts to $\Hom_\ct(\SS_{<d+1},X)=0$ and $\Hom_\ct(\SS_{>0},X)=0$, which, by $\SS=\SS_{>0}*\SS_{\le0}$ (Proposition~\ref{from silting to co-t-structure}), is equivalent to $X\in\SS^{\perp_\ct}$.
\end{proof}

For $X\in\ct$, we have a triangle
\begin{equation}\label{t-structure new}
\xymatrix{\sigma^{\le0}X\ar[r]^{a_X}&X\ar[r]^{b_X}&\sigma^{\geq 1}X\ar[r]^{c_X}&(\sigma^{\le0}X)[1]}
\end{equation}
in $\ct$ such that $\sigma^{\le0}X\in\ct^{\le0}$ and
$\sigma^{\geq 1}X\in\ct^{\geq 1}\subset\ct^{\fd}$.

\begin{lemma}\label{truncation new}
Let $X\in\cz$.
Then $\sigma^{\geq 1}X\in \ct^{\fd}\cap\cz$
and $\sigma^{\le0}X\in\cz$.
\end{lemma}

\begin{proof}
Since $\cp\subset\cm$, we have by the definition of $\ct^{\geq 1}$ that
\begin{eqnarray}\label{i le0 new}
\Hom_{\ct}(\cp,(\sigma^{\geq 1}X)[i])=0\ \mbox{ for any }\ i\le0,
\end{eqnarray}
and by the definition of $\ct^{\leq 0}$ that
\begin{eqnarray}\label{i ge1 new} 
\Hom_{\ct}(\cp,(\sigma^{\le0}X)[i])=0\ \mbox{ for any }\ i\ge1.
\end{eqnarray}
Applying $\Hom_{\ct}(\cp,-)$ to the triangle \eqref{t-structure new},
we obtain an exact sequence
\[
\Hom_{\ct}(\cp,X[i])\to\Hom_{\ct}(\cp,(\sigma^{\geq 1}X)[i])\to
\Hom_{\ct}(\cp,(\sigma^{\le0}X)[i+1]).
\]
Assume $i\geq 1$. Then the left term vanishes because $X\in\cz$ and the right term vanishes due to \eqref{i ge1 new}.
Thus we have $\Hom_{\ct}(\cp,(\sigma^{\geq 1}X)[i])=0$ for any $i\ge1$. Combined with \eqref{i le0 new}, this yields 
$\sigma^{\geq 1}X\in \ct^{\fd}\cap\SS^{\perp_\ct}=\cu^{\fd}$. By Lemma~\ref{l:fd-objects-of-Z}, $\sigma^{\geq 1}X\in\ct^{\fd}\cap \cz$.

Moreover, $(\sigma^{\geq 1}X)[-1]$ belongs to $\cu^{\fd}=\ct^{\fd}\cap\cz$. Since $\cz$ is closed under extensions and $X\in\cz$, the triangle \eqref{t-structure new} shows $\sigma^{\le0}X\in\cz$.
\end{proof}

\begin{proof}[Proof of Theorem~\ref{p:silting-reduction-is-3cy}]
By Lemma~\ref{l:fd-objects-of-Z}, the category $\ct^{\fd}\cap \cz$ is left and right orthogonal to $\cp$, thus it can be viewed as a full subcategory of $\frac{\cz}{[\cp]}$.
Moreover, it follows from Lemma~\ref{l:fd-objects-of-Z} that on $\ct^{\fd}\cap \cz$ there is a natural isomorphism $\langle 1\rangle\simeq [1]$. Therefore $\ct^{\fd}\cap\cz$ is naturally a triangulated subcategory of $\frac{\cz}{[\cp]}$. 
Thanks to the equivalence $G$, to prove the theorem it suffices to show that the statements (a), (b) and (c) hold for the triple $(\frac{\cz}{[\cp]},\ct^{\fd}\cap \cz,\frac{\cm}{[\cp]})$.

(a) The category $\cz$ is a full subcategory of $\ct$ which is closed under direct summands. Thus it is a Hom-finite and Krull--Schmidt, so is the additive quotient $\frac{\cz}{[\cp]}$. 

(b) Since on $\ct^{\fd}\cap \cz$ there is a natural isomorphism $\langle 1\rangle\simeq [1]$, it follows that for $X\in\ct^{\fd}\cap\cz$ and $Y\in\frac{\cz}{[\cp]}$ we have 
$\Hom_\ct(X,\cp)\simeq D\Hom_\ct(\cp,X[d+1])=0$ and $\Hom_\ct(\cp,X[d+1])=0$. Therefore we have
bifunctorial isomorphisms
\begin{eqnarray*}
D\Hom_{\frac{\cz}{[\cp]}}(X,Y)&=&D\Hom_\cz(X,Y)
\simeq\Hom_\cz(Y,X[d+1])
=\Hom_{\frac{\cz}{[\cp]}}(Y,X[d+1])\\
&\simeq&\Hom_{\frac{\cz}{[\cp]}}(Y,X\langle d+1\rangle).
\end{eqnarray*}

(c) By Theorem~\ref{t:silting-reduction}, $\frac{\cm}{[\cp]}\subset\frac{\cz}{[\cp]}$ is a silting subcategory.
By Lemma~\ref{no Homs}, to prove that $(\frac{\cm}{[\cp]}\langle<\hspace{-3pt}0\rangle^{\perp_{\frac{\cz}{[\cp]}}},\frac{\cm}{[\cp]}\langle>\hspace{-3pt}0\rangle^{\perp_{\frac{\cz}{[\cp]}}})=(\cm\langle<\hspace{-3pt}0\rangle^{\perp_{\frac{\cz}{[\cp]}}},\cm\langle>\hspace{-3pt}0\rangle^{\perp_{\frac{\cz}{[\cp]}}})$ is a t-structure it suffices to prove $\frac{\cz}{[\cp]}=(\cm\langle<\hspace{-3pt}0\rangle^{\perp_{\frac{\cz}{[\cp]}}})*(\cm\langle\geq\hspace{-3pt}0\rangle^{\perp_{\frac{\cz}{[\cp]}}})$.
Let $X\in\cz$. 
By Theorem~\ref{triangle structure of Z/P}(b), the triangle \eqref{t-structure new} induces a triangle in $\frac{\cz}{[\cp]}$
\begin{equation}\label{eq:canonical-triangle-reduction}
\xymatrix{\sigma^{\le0}X\ar[r]^{\underline{a}_X}&X\ar[r]^{\underline{b}_X}&\sigma^{\geq 1}X\ar[r]&\sigma^{\le0}X\langle 1\rangle}.
\end{equation}
We only have to show that $\sigma^{\leq 0}X\in \cm\langle<\hspace{-3pt}0\rangle^{\perp_{\frac{\cz}{[\cp]}}}$ and $\sigma^{\geq 1}X\in \cm\langle\geq\hspace{-3pt}0\rangle^{\perp_{\frac{\cz}{[\cp]}}}$.
We know that $\sigma^{\geq 1}X\in \ct^{\fd}\cap\cz$ and $\sigma^{\le0}X\in\cz$ hold by Lemma \ref{truncation new}.

Fix $i\ge0$. Then we have $\cm\langle i\rangle\in
\cp*\cdots*\cp[i-1]*\cm[i]$ by the construction of $\langle i\rangle$.
This implies $\Hom_{\ct}(\cm\langle i\rangle,\ct^{\ge1})=0$.
Hence $\ct^{\ge1}\cap\cz\ni\sigma^{\geq 1}X$ is contained in
$\cm\langle\geq\hspace{-3pt}0\rangle^{\perp_{\frac{\cz}{[\cp]}}}$.

Fix $i>0$. Then we have $\cm\langle1-i\rangle\in \cm[1-i]*\cp[2-i]*\cdots*\cp$ by the construction of $\langle1-i\rangle$.
This implies $\Hom_{\ct}(\cm\langle1-i\rangle[-1],\ct^{\le0})=0$.
Further, for any $M\in\cm$, we have a triangle
\[
\xymatrix{
M\langle1-i\rangle[-1]\ar[r] & M\langle-i\rangle\ar[r]^(0.55)b & P\ar[r]^(0.38)a & M\langle1-i\rangle
}
\]
with a right $\cp$-approximation $a$.
Applying $\Hom_{\ct}(-,\ct^{\le0})$ to this triangle, we have that the map $\Hom_{\ct}(P,\ct^{\le0})\to\Hom_{\ct}(M\langle-i\rangle,\ct^{\le0})$ is surjective.
Hence $\Hom_{\frac{\cz}{[\cp]}}(\cm\langle-i\rangle,\ct^{\le0}\cap\cz)=0$, and
$\ct^{\le0}\cap\cz\ni\sigma^{\leq 0}X$ is contained in
$\cm\langle<\hspace{-3pt}0\rangle^{\perp_{\frac{\cz}{[\cp]}}}$.

Consequently, $(\cm\langle<\hspace{-3pt}0\rangle^{\perp_{\frac{\cz}{[\cp]}}},\cm\langle>\hspace{-3pt}0\rangle^{\perp_{\frac{\cz}{[\cp]}}})$ forms a t-structure on $\frac{\cz}{[\cp]}$.
Finally, if $X\in\cm\langle\ge\hspace{-3pt}0\rangle^{\perp_{\frac{\cz}{[\cp]}}}$, the triangle \eqref{eq:canonical-triangle-reduction} shows that $X$ is isomorphic to $\sigma^{\geq 1}X$ and hence lies in $\cu^{\fd}=\ct^{\fd}\cap\cz$. Consequently, $\cm\langle>\hspace{-3pt}0\rangle^{\perp_{\frac{\cz}{[\cp]}}} =(\cm\langle\ge\hspace{-3pt}0\rangle^{\perp_{\frac{\cz}{[\cp]}}})\langle1\rangle$ is contained in $\cu^{\fd}$.

Finally, that $\frac{\cm}{[\cp]}$ is a dualizing $k$-variety follows from the following elementary observation. This completes the proof.
\end{proof}

\begin{proposition}
Let $\cm$ be a dualizing $k$-variety and $\cp$ a functorially finite subcategory of $\cm$. Then $\frac{\cm}{[\cp]}$ is again a dualizing $k$-variety. 
\end{proposition}

\begin{proof}
Since $\cp$ is a functorially finite subcategory of $\cm$, it follows that the representable functors of $\frac{\cm}{[\cp]}$ (respectively, $(\frac{\cm}{[\cp]})^{\op}$) are finitely presented as $\cm$-modules (respectively, as $\cm^{\op}$-modules). One checks that an $\frac{\cm}{[\cp]}$-module (respectively, $(\frac{\cm}{[\cp]})^{\op}$-module) is finitely presented as an $\frac{\cm}{[\cp]}$-module (respectively, $(\frac{\cm}{[\cp]})^{\op}$-module) if and only if it is finitely presented as an $\cm$-module (respectively, $\cm^{\op}$-module). Therefore we have a commutative diagram
\[
\xymatrix{
\mod\frac{\cm}{[\cp]}\ar[d]\ar@{<->}[rr]^D && \mod (\frac{\cm}{[\cp]})^{\op}\ar[d]\\
\mod\cm\ar@{<->}[rr]^D && \mod \cm^{\op},
}
\]
showing that $\frac{\cm}{[\cp]}$ is a dualizing $k$-variety.
\end{proof}

\subsection{The Amiot--Guo--Keller cluster category of a Calabi--Yau triple}\label{ss:AGK-cluster-category}
Assume that  $(\ct,\ct^{\fd},\cm)$ is a $(d+1)$-Calabi--Yau triple. We keep the notation in Section~\ref{ss:setup}.
Consider the triangle quotient 
\[
\cc:=\ct/\ct^{\fd},
\]
which we call \emph{Amiot--Guo--Keller cluster category} of $\ct$. 
Let $\pi\colon\ct\rightarrow\cc$ denote the canonical projection functor. 
We define a full subcategory $\cf$ of $\ct$ by
\[\cf:=\ct_{\geq1-d}\cap\ct_{\leq 0}\stackrel{{\rm Prop. \ref{from silting to co-t-structure}(b)}}{=}\cm*\cm[1]*\cdots*\cm[d-1].\]

Now we give the following generalisation of fundamental results due to Amiot and Guo \cite{Amiot09,Guolingyan11a} to our setting of $(d+1)$-Calabi--Yau triples. 
In particular, the statement (b) says that $\cf$ is a fundamental domain of $\cc$ in $\ct$. We observe that a hidden key point of the proofs in \cite{Amiot09,Guolingyan11a} is the existence of right and left adjacent t-structures in (CY3) and (CY3$^{\op}$). This motivates our study in Section~\ref{section:adjacent t-structures} and enables us to make the generalisation.

\begin{theorem}\label{amiot}
\begin{itemize}
\item[(a)] The category $\cc$ is a $d$-Calabi--Yau triangulated category. 
\item[(b)] The functor $\pi\colon\ct\to\cc$ restricts to an equivalence $\cf\to\cc$ of additive categories.
\item[(c)] $\pi(\cm)$ is a $d$-cluster-tilting subcaegory of $\cc$ and $\pi\colon\cm\to\pi(\cm)$ is an equivalence.
\end{itemize}
\end{theorem}

The following proposition will play an important role in the proof of Theorem~\ref{amiot} and Theorem~\ref{main theorem new}.

\begin{proposition}\label{fundamental domain}
The functor $\pi\colon\ct\to\cc$ induces a bijection (respectively, injection) $\Hom_{\ct}(U,V)\rightarrow\Hom_{\cc}(U,V)$ for any $U\in\ct_{\le0}$ and $V\in\ct_{\ge1-d}$ (respectively, $V\in\ct_{\ge-d}$). Consequently, it restricts to a fully faithful functor $\cf\to\cc$. \end{proposition}

In particular, for $M,N\in\cm$, we have isomorphisms $\Hom_\ct(M,N[i])\simeq\Hom_\cc(M,N[i])$ for all $i\leq d-1$.
To prove this proposition we need the following lemma.

\begin{lemma}\label{l:adjusting-roof}
Let $X\in\ct_{\leq 0}$ and $Y\in\ct$. Then any morphism in $\Hom_\cc(X,Y)$ has a representative of the form $X\xleftarrow{s}Z\xrightarrow{f}Y$ such that the cone of $s$ belongs to $\ct_{\leq 0}\cap\ct^{\fd}$.
\end{lemma}

\begin{proof}
Any morphism $X\to Y$ in $\cc$ can be written as
$X\xleftarrow{s}Z\xrightarrow{f}Y$ such that there exists a triangle
\[\xymatrix{Z\ar[r]^s&X\ar[r]^t&W\ar[r]&Z[1]}\]
with $W\in\ct^{\fd}$. Recall that $\ct_{\leq 0}=\ct^{\leq 0}$. Thus $t$ factors through
$\sigma^{\leq 0}W\to W$ since $\Hom_{\ct}(X,\sigma^{\geq 1}W)=0$.
We obtain a commutative diagram of triangles:
\[\xymatrix{
&&\sigma^{\geq 1}W\\
Z\ar[r]^s&X\ar[r]^t&W\ar[r]\ar[u]&Z[1]\\
Z'\ar[r]^{sh}\ar[u]^h&X\ar[r]\ar@{=}[u]&\sigma^{\le0}W\ar[r]\ar[u]&Z'[1]\ar[u]
}\]
Because the cone $\sigma^{\leq 0}W$ of $sh$ belongs to $\ct_{\leq 0}\cap\ct^{\fd}$ by Lemma~\ref{l:t-str-restricts-to-fd-objects}, the morphism
$X\xleftarrow{s}Z\xrightarrow{f}Y$ is equivalent to $X\xleftarrow{sh}Z'\xrightarrow{fh}Y$, so the assertion follows.
\end{proof}

\begin{proof}[Proof of Proposition \ref{fundamental domain}]
Let $U\in\ct_{\leq 0}$ and $V\in\ct_{\geq-d}$. 

First we show that $\Hom_{\ct}(U,V)\to\Hom_{\cc}(U,V)$ is injective.
Assume that $f\in\Hom_{\ct}(U,V)$ becomes zero in $\cc$.
Then it factors through some $W\in\ct^{\fd}$ (by, for example, \cite[Lemma 2.1.26]{Neeman01b}), and further through $\sigma^{\le0}W$ because $U\in\ct_{\leq 0}$.
By the relative $(d+1)$-Calabi--Yau property, we have
\[\Hom_{\ct}(\sigma^{\le0}W,V)\simeq D\Hom_{\ct}(V,(\sigma^{\le0}W)[d+1])=0\]
as $V\in\ct_{\ge-d}$. Thus $f$ must be zero.

Next we show that $\Hom_{\ct}(U,V)\to\Hom_{\cc}(U,V)$ is surjective if $V\in\ct_{\geq1-d}$.
By Lemma~\ref{l:adjusting-roof}, a morphism in $\Hom_\cc(U,V)$ has a representative of the form  $U\xleftarrow{s}Y\xrightarrow{f}V$ such that the cone $W$ of $s$ belongs to $\ct_{\leq 0}\cap\ct^{\fd}$.
We have an exact sequence
\[
\Hom_{\ct}(U,V)\xrightarrow{s}\Hom_{\ct}(Y,V)\to\Hom_{\ct}(W[-1],V).
\]
As $W[-1]\in\ct^{\fd}$, we can apply the relative $(d+1)$-Calabi--Yau property to obtain
\[
\Hom_{\ct}(W[-1],V)\simeq D\Hom_{\ct}(V,W[d])=0.
\]
The last equality holds because $V\in\ct_{\ge1-d}$ and $W[d]\in\ct_{\leq-d}$. So there exists $g\in\Hom_\ct(U,V)$ such that $f=gs$, and hence $U\xleftarrow{s}Y\xrightarrow{f}V$ is equivalent to $U\xrightarrow{g}V$. It follows that $\Hom_{\ct}(U,V)\to\Hom_{\cc}(U,V)$ is surjective.
\end{proof}

We also need the following observation.

\begin{lemma}\label{image of sigma le0}
We have $\sigma^{\le0}(\ct_{\ge1-d})\subset\cf$.
\end{lemma}

\begin{proof}
We need to show $\sigma^{\le0}X\in\ct_{\ge1-d}$, that is, $\Hom_{\ct}(\sigma^{\le0}X,\cm[\ge\hspace{-3pt} d])=0$.
Consider the triangle 
\[
\xymatrix{\sigma^{\le0}X\ar[r]& X\ar[r] & \sigma^{\ge1}X\ar[r] & (\sigma^{\le0}X)[1].}
\]
Applying $\Hom_\ct(-,\cm[\ge\hspace{-3pt} d])$, we have an exact sequence
\[\Hom_\ct(X,\cm[\ge\hspace{-3pt} d])\to\Hom_\ct(\sigma^{\le 0}X,\cm[\ge\hspace{-3pt} d])\to\Hom_\ct((\sigma^{\ge1}X)[-1],\cm[\ge\hspace{-3pt} d]).\]
Since $X\in\ct_{\ge1-d}$, we have $\Hom_\ct(X,\cm[\ge\hspace{-3pt} d])=0$.
Moreover $\Hom_\ct((\sigma^{\ge1}X)[-1],\cm[\ge\hspace{-3pt} d])\simeq D\Hom_\ct(\cm,(\sigma^{\ge1}X)[\le\hspace{-3pt}0])=0$.
Thus the assertion follows.
\end{proof}

Now we are ready to prove Theorem~\ref{amiot}.

\begin{proof}[Proof of Theorem \ref{amiot}]
(b) The functor $\cf\rightarrow\cc$ is fully faithful by Proposition \ref{fundamental domain}. It remains to show that it is dense. Let $X$ be any object of $\cc$ and view it as an object of $\ct$.
By (CY3) and (CY3$^{\op}$), we have t-structures $(\cm[<\hspace{-3pt}0]^{\perp_{\ct}},\cm[>\hspace{-3pt}0]^{\perp_{\ct}})$ and $({}^{\perp_{\ct}}\cm[<\hspace{-3pt}d],{}^{\perp_{\ct}}\cm[>\hspace{-3pt}d])$ on $\ct$ satisfying $\cm[>\hspace{-3pt}0]^{\perp_{\ct}}\subset\ct^{\fd}$ and ${}^{\perp_{\ct}}\cm[<\hspace{-3pt}d]\subset\ct^{\fd}$.
The second t-structure gives a triangle
\[\xymatrix{Y\ar[r]& X\ar[r]& Z\ar[r]& Y[1]}\]
with $Y\in{}^{\perp_{\ct}}\cm[\leq \hspace{-3pt}d]$ and $Z\in{}^{\perp_{\ct}}\cm[>\hspace{-3pt}d]=\ct_{\ge1-d}$. The first t-structure gives a triangle
\[\xymatrix{\sigma^{\le0}Z\ar[r]& Z\ar[r]&\sigma^{\ge1}Z\ar[r]&(\sigma^{\le0}Z)[1]}\]
with $\sigma^{\le0}Z\in\cm[<\hspace{-3pt}0]^{\perp_{\ct}}$ and $\sigma^{\ge1}Z\in\cm[\ge\hspace{-3pt}0]^{\perp_{\ct}}$.
Then $\sigma^{\le0}Z\in\sigma^{\le0}(\ct_{\ge1-d})\subset\cf$ holds by Lemma \ref{image of sigma le0}. Since both $Y$ and $\sigma^{\ge1}Z$ belong to $\ct^{\fd}$, we have $X\simeq Z\simeq\sigma^{\le0}Z\in\cf$ in $\cc$.
Thus the assertion follows.

(a) First, by (b), the category $\cc$ is Hom-finite.

Secondly, we show that $\cc$ is $d$-Calabi--Yau. Let $X$ and $Y$ be objects of $\ct$.
Recall that $(\ct_{\geq 0},\ct_{\leq 0})$ is a bounded co-t-structure on $\ct$. It follows that there exists an integer $i$ such that $Y$ belongs to $\ct_{\geq i}$.
Now consider the triangle
\[
\xymatrix{
\sigma^{\leq i-1}X\ar[r] & X\ar[r] & \sigma^{\geq i}X\ar[r] &(\sigma^{\leq i-1}X)[1].
}
\]
Because $\sigma^{\leq i-1}X\in \ct^{\leq i-1}=\ct_{\leq i-1}$, we have $\Hom_\ct(Y,\sigma^{\leq i-1}X)=0$. It follows that the induced homomorphism $\Hom_{\ct}(Y,X)\rightarrow\Hom_{\ct}(Y,\sigma^{\geq i}X)$ is injective.
So the morphism $X\rightarrow \sigma^{\geq i}X$ is a local $\ct^{\fd}$-envelope of $X$ relative to $Y$ in the sense of \cite[Definition 1.2]{Amiot09}.
Therefore by \cite[Lemma 1.1, Theorem 1.3 and Proposition 1.4]{Amiot09} we see that $\cc$ is $d$-Calabi--Yau.

(c) As all $\cm[i]$, $0\leq i\leq d-1$ belong to $\cf$, we have by Proposition \ref{fundamental domain} that $\pi\colon\cm\to\pi(\cm)$ is an equivalence, and $\Hom_\cc(\cm,\cm[i])\simeq \Hom_\ct(\cm,\cm[i])=0$ for $1\leq i\leq d-1$, i.e. $\cm$ is $d$-rigid.
Since $\cf=\cm*\cm[1]*\cdots*\cm[d-1]$ by definition and $\pi\colon\cf\to\cc$ is dense, we have $\cc=\pi(\cm)*\pi(\cm)[1]*\cdots*\pi(\cm)[d-1]$. Thus $\pi(\cm)$ is a $d$-cluster-tilting subcategory of $\cc$.
\end{proof}

We end this subsection with the observation below, where the $d=2$ case of part (b) is due to Keller and Nicol\'as \cite{KellerNicolas11} in the algebraic case, see also \cite[Theorem 4.5]{BY}.
Let
\[\silt^{\cf}\ct:=\{\cn\in\silt\ct\mid\cn\subset\cf\}.\]
Let $\dctilt\cc$ be the class of $d$-cluster-tilting subcategories of $\cc$, where we identify two $d$-cluster-tilting subcategories $\NN$ and $\cn'$ of $\cc$ when $\add\NN=\add\cn'$.

\begin{corollary}\label{2term-silting and cluster-tilting}
If $\cm=\add M$ for some silting object $M$ of $\ct$, then the following statements hold.
\begin{itemize}
\item[(a)] The functor $\pi\colon\ct\to\cc$ gives a map $\pi\colon\silt\ct\to\dctilt\cc$.
\item[(b)] The map in (a) restricts to an injection $\pi\colon\silt^{\cf}\ct\to\dctilt\cc$, which is a bijection if $d=1$ or $d=2$.
\end{itemize}
\end{corollary}

\begin{proof}
For any $\cn\in\silt\ct$, it follows from Remark \ref{idependence of M} that $(\ct,\ct^{\fd},\cn)$ is a $(d+1)$-Calabi--Yau triple. 
Thus, by Theorem~\ref{amiot}, $\pi(\cn)$ is a $d$-cluster-tilting subcategory of $\cc$. In this way, we obtain a map $\pi\colon\silt\ct\to\dctilt\cc$.
Since $\pi\colon\cf\to\cc$ is fully faithful by Proposition \ref{fundamental domain}, the induced map $\pi\colon\silt^{\cf}\ct\to\dctilt\cc$ is injective.

We show that it is surjective for $d=1$ and $d=2$. For $d=1$ this is true, since we have $\silt^{\cf}\ct=\{\cm\}$ and $\dctilt\cc=\{\pi(\cm)\}$. Next assume $d=2$.
For a subcategory $\cn$ of $\cf$, assume that $\pi(\cn)$ is a $2$-cluster-tilting subcategory of $\cc$.
Then $\cn$ is a presilting subcategory of $\ct$ since $\Hom_{\ct}(\cn,\cn[\ge\hspace{-3pt}2])=0$ by $\cn\subset\cf$ and
$\Hom_{\ct}(\cn,\cn[1])\to\Hom_{\cc}(\cn,\cn[1])$ is injective by
Proposition \ref{fundamental domain}.
Using Bongartz completion \cite[Proposition 4.2]{IJY}, there exists $\cn'\in\silt^{\cf}\ct$ containing $\cn$.
Since $\pi(\cn')$ is a $2$-cluster-tilting subcategory of $\cc$ containing $\pi(\cn)$, we have $\pi(\cn)=\pi(\cn')$. Therefore $\cn=\cn'$ holds.
\end{proof}

\begin{remark}
Assume $d=2$, and let $M$ be a silting object in $\ct$ and $\Lambda:=\End_{\ct}(M)$.
It is shown in \cite{AIR} that we have a bijection $\twosilt\Lambda\to\twoctilt\cc$, where $\twosilt\Lambda$ denotes the set of 2-term silting objects in $\ck^{\bo}(\proj\Lambda)$.
Thus there is a bijective map $\silt^{\cf}\ct\to\twosilt\Lambda$ making the following diagram of bijective maps commutative
\[\xymatrix{
\silt^{\cf}\ct\ar[rr]\ar[rd]^{\pi}&&\twosilt\Lambda\ar[ld]\\
&\twoctilt\cc
}\]
Under the assumption that $\ct$ is an algebraic triangulated category this is given in \cite{BY}. Note, however, that in this case there is a triangle functor $\ct\rightarrow\ck^{\bo}(\proj\Lambda)$, which induces a bijective map $\silt^{\cf}\ct\to\twosilt\Lambda$ making the above diagram commutative, see \cite[Proposition A.3]{BY} (and Theorem A.7 of the arXiv version of \cite{BY}). In the general setting the  triangle functor $\ct\rightarrow\ck^{\bo}(\proj\Lambda)$ and the direct definition of the map $\silt^{\cf}\ct\to\twosilt\Lambda$ are not available.
\end{remark}

We do not know if the map $\pi\colon\silt^{\cf}\ct\to\dctilt\cc$ in Corollary~\ref{2term-silting and cluster-tilting}(b) is bijective for $d>2$. We conjecture that this is the case.

\begin{conjecture}
The map $\pi\colon\silt^{\cf}\ct\to\dctilt\cc$ in Corollary~\ref{2term-silting and cluster-tilting}(b) is bijective for all $d\geq 1$.
\end{conjecture}

\subsection{Silting reduction induces Calabi--Yau reduction}\label{ss:silting-red-vs-cy-red} 
Let $(\ct,\ct^{\fd},\cm)$ be a $(d+1)$-Calabi--Yau triple, as in Section~\ref{ss:setup}. Let $\cp$ be a functorially finite subcategory of $\cm$.

By Theorem \ref{amiot}, $\cc=\ct/\ct^{\fd}$ is a $d$-Calabi--Yau triangulated category and $\pi(\cm)$ is a $d$-cluster-tilting object of $\cc$.
In particular, $\pi(\cp)$ is $d$-rigid. Here $\pi\colon\ct\rightarrow\cc$ is the canonical projection functor. By abuse of notation, we will write $\cm$ and $\cp$ for $\pi(\cm)$ and $\pi(\cp)$.

Analogous to \eqref{P and Z}, we define a subcategory of $\cc$ by 
\[
\cz':={}^{\perp_{\cc}}(\pi(\cp)[1]*\pi(\cp)[2]*\cdots*\pi(\cp)[d-1]).\]
Thus, we can form the Calabi--Yau reduction as explained in Section~\ref{section: Mutation pair}:
\[
\cc_{\cp}:=\frac{\cz'}{[\pi(\cp)]}.
\]
By Theorem~\ref{t:cy-reduction}, the subcategory $\frac{\pi(\cm)}{[\pi(\cp)]}$ in $\cc_{\cp}$ is $d$-cluster-tilting, and by Proposition~\ref{fundamental domain}, we have an equivalence
\begin{equation}\label{End CP}
\frac{\pi(\cm)}{[\pi(\cp)]}\simeq\frac{\cm}{[\cp]}.
\end{equation}

On the other hand, let $\SS:=\thick\cp$, $\cu:=\ct/\SS$ and $\rho\colon\ct\to\cu$ the canonical projection. We consider $\cu^{\fd}:=\ct^{\fd}\cap\SS^{\perp_\ct}$ as a full subcategory of $\cu$.
Then $(\cu,\cu^{\fd},\rho(\cm))$ is a relative $(d+1)$-Calabi--Yau triple by Theorem~\ref{p:silting-reduction-is-3cy}, and  the triangle quotient 
\[
\cu/\cu^{\fd}
\]
is a $d$-Calabi--Yau triangulated category by Theorem~\ref{amiot}.
Let $\pi_{\cu}\colon\cu\to\cu/\cu^\fd$ be the canonial projection.
Then the subcategory $\pi_{\cu}(\rho(\cm))$ in $\cu/\cu^{\fd}$ is $d$-cluster-tilting, and by Proposition~\ref{fundamental domain} and Theorem~\ref{equivalence}, we have equivalences 
\begin{equation}\label{End U Ufd}
\pi_{\cu}(\rho(\cm))\simeq\rho(\cm)\simeq\frac{\cm}{[\cp]}.
\end{equation}

Therefore, we obtain two $(d+1)$-Calabi--Yau triangulated categories, $\cc_{\cp}$ and $\cu/\cu^{\fd}$, and they have $d$-cluster-tilting subcategories, which are equivalent to each other.
The following main result asserts that these two triangulated categories are equivalent.

\begin{theorem}\label{main theorem new}
The two categories $\cc_{\cp}$ and $\cu/\cu^{\fd}$ are triangle equivalent.\end{theorem}

In this sense, we say that the AGK cluster category construction Theorem \ref{amiot}
takes the silting reduction of $\ct$ with respect to $\cp$ to the Calabi--Yau reduction of $\cc$ with respect to $\pi(\cp)$.

\begin{remark}
Let $(Q,W)$ be a quiver with potential and $\Gamma=\Gamma(Q,W)$ be its complete Ginzburg dg algebra, see \cite{DerksenWeymanZelevinsky08,Ginzburg06,KellerYang11}. Assume that $H^0(\Gamma)$ is finite-dimensional. Then the triple $(\per(\Gamma),\cd_{\fd}(\Gamma),\Gamma)$ is a 3-Calabi--Yau triple.
The triangle quotient 
$$\cc(Q,W)=\per(\Gamma)/\cd_{\fd}(\Gamma)$$ is called the \emph{cluster category} of $(Q,W)$. Let $i$ be a vertex of $Q$, $e=e_i$ be the trivial path at $i$, and $(Q',W')$ be the quiver with potential obtained from $(Q,W)$ by deleting the vertex $i$. It is stated in \cite[Theorem 7.4]{Keller11} that there is a triangle equivalence between the Calabi--Yau reduction of $\cc(Q,W)$ with respect to $e_i\Gamma$ and the cluster category $\cc(Q',W')$ of $(Q',W')$. In conjunction with \cite[Corollary 7.3]{Keller11} our Theorem~\ref{main theorem new} provides an alternative proof to this statement.
\end{remark}

We start the proof of Theorem~\ref{main theorem new} with two lemmas.

\begin{lemma}\label{compatibility of approximation new}
For any $X\in\cz$ and for $i\leq d-1$, the map
\begin{equation}\label{compare approximation new}
\Hom_{\ct}(X,\cp[i])\to\Hom_{\cc}(X,\cp[i])
\end{equation}
is bijective. In particular, $\Hom_\cc(X,\cp[i])=0$ for $1\leq i\leq d-1$.
\end{lemma}

\begin{proof}
Consider the triangle \eqref{t-structure new}, which induces a commutative diagram for $i\leq d-1$
\[\xymatrix{
\Hom_{\ct}(X,\cp[i])\ar[r]^{a_X}\ar[d]&\Hom_{\ct}(\sigma^{\le0}X,\cp[i])\ar[d]\\
\Hom_{\cc}(X,\cp[i])\ar[r]^{a_X}&\Hom_{\cc}(\sigma^{\le0}X,\cp[i]).
}\]
The upper map is bijective since $\sigma^{\geq 1}X\in\cu^{\fd}\subset{}^{\perp_\ct}\SS$
holds by Lemma \ref{truncation new} and Lemma~\ref{l:fd-objects-of-Z}, 
and the lower map is bijective since $a_X:\sigma^{\le0}X\rightarrow X$ becomes an isomorphism in $\cc$.
Further, since $\sigma^{\leq 0}X\in\ct^{\leq 0}=\ct_{\leq 0}$ and $\cp[i]\subset\ct_{\geq 1-d}$, the right map is bijective by Proposition~\ref{fundamental domain}. The bijectivity of the left map follows immediately.

As $X\in\cz$, we have $\Hom_\ct(X,\cp[>\hspace{-3pt}0])=0$. In conjunction with the first statement, this implies the second statement.
\end{proof}

\begin{lemma}\label{density}
The functor $\pi\colon\ct\to\cc$ induces a dense functor $\cz\to\cz'$.
\end{lemma}

\begin{proof}
By Lemma~\ref{compatibility of approximation new}, $\pi$ gives a functor $\cz\to\cz'$.
We need to show that this is dense.

Fix any $Y\in\cz'$. By Theorem~\ref{amiot} (b),
there exists $X\in\cf=\ct_{\ge 1-d}\cap\ct_{\le0}$ such that $\pi(X)\simeq Y$.
Since $\cp\subset\cm$, we have $\Hom_{\ct}(\cp,X[\ge\hspace{-3pt}1])=0$ and $\Hom_{\ct}(X,\cp[\ge\hspace{-3pt}d])=0$.
By Proposition~\ref{fundamental domain}, we have $\Hom_{\ct}(X,\cp[i])\simeq\Hom_{\cc}(Y,\cp[i])=0$ for $1\leq i\leq d-1$.
Thus $X\in\cz$ and the assertion follows.
\end{proof}

Therefore the functor $\pi\colon\ct\to\cc$
induces additive functors $\cz\to\cz'$ and $\cp\to\pi(\cp)$, and further induces an additive functor
\begin{equation}\label{key functor new}
\tilde{\pi}\colon\cu\simeq\frac{\cz}{[\cp]}\longrightarrow\cc_{\cp}=\frac{\cz'}{[\pi(\cp)]}.
\end{equation}
We observed in Sections \ref{subsection: Triangle equivalence} and \ref{section: Mutation pair} that both categories $\frac{\cz}{[\cp]}$ and $\frac{\cz'}{[\pi(\cp)]}$
have structures of triangulated categories. 
Now we show the following.

\begin{proposition}\label{functor is triangle new}
The functor $\tilde{\pi}\colon\cu\rightarrow\cc_{\cp}$ is a triangle functor which is dense.
\end{proposition}

\begin{proof}
By Lemma \ref{compatibility of approximation new},
the image of a left $\cp$-approximation in $\cz$
gives a left $\pi(\cp)$-approximation in $\cz'$.
Thus the functor commutes with shifts.

Next we show that the functor sends triangles to triangles.
The triangles in $\frac{\cz}{[\cp]}$ are defined by the commutative diagram \eqref{define a triangle} in Theorem \ref{triangle structure of Z/P}. The image of \eqref{define a triangle} in $\cc$ is also a commutative
diagram of triangles with a left $\pi(\cp)$-approximation $\iota_X$ by
Lemma \ref{compatibility of approximation new}.
Thus $X\xrightarrow{\overline{f}}Y\xrightarrow{\overline{g}}Z
\xrightarrow{\overline{a}}X\langle1\rangle$ is a triangle in $\frac{\cz'}{[\pi(\cp)]}$.
Thus the assertion follows.

The functor $\tilde{\pi}\colon\cu\rightarrow\cc_{\cp}$ is dense by Lemma \ref{density}.
\end{proof}

Now we are ready to prove Theorem \ref{main theorem new}.

\begin{proof}[Proof of Theorem \ref{main theorem new}]
Since $\pi(\ct^{\fd})=0$ and $\cu^{\fd}\subset \ct^{\fd}$, we have $\tilde{\pi}(\cu^{\fd})=0$. Therefore $\tilde{\pi}$ induces a triangle
functor $\pi'\colon\cu/\cu^{\fd}\to\cc_{\cp}$. It remains to show that $\pi'$ is an equivalence.
Tracing the construction of $\pi'$, we see that $\pi'$ sends the $d$-cluster-tilting subcategory
$\pi_{\cu}(\rho(\cm))$ of $\cu/\cu^{\fd}$ to the $d$-cluster-tilting subcategory $\frac{\pi(\cm)}{[\pi(\cp)]}$ of $\cc_{\cp}$.
Moreover, we have equivalences of categories
\[
\pi_{\cu}(\rho(\cm))\stackrel{\eqref{End U Ufd}}{\simeq}
\frac{\cm}{[\cp]}\stackrel{\eqref{End CP}}{\simeq}\frac{\pi(\cm)}{[\pi(\cp)]},
\]
whose composition is induced by $\pi'$.
Thus the triangle functor $\pi'\colon\cu/\cu^{\fd}\to\cc_{\cp}$ is an equivalence by Proposition \ref{cluster-Beilinson}.
\end{proof}

\section{Conjectures of Auslander--Reiten and Tachikawa}\label{section: homological conjectures}

In this section, we discuss the relationship between silting theory and the conjecture of Tachikawa and that of Auslander--Reiten.

\medskip

Let $k$ be a field and $A$ be a finite-dimensional $k$-algebra and let $n$ be the number
of pairwise non-isomorphic simple $A$-modules.
Motivated by Tachikawa's study \cite{Tachikawa73} on the famous Nakayama conjecture, Auslander and Reiten proposed the following conjecture:

\medskip\noindent
{\bf The Auslander--Reiten Conjecture} (\cite{AuslanderReiten75})
\emph{If $X\in\mod A$ satisfies $\Ext^i_A(X,X\oplus A)=0$ for all $i>0$,
then $X$ is a projective $A$-module.}

\medskip
Now we pose the following conjectures in the context of silting theory.

\begin{conjecture}\label{Tachikawa 1.5}
$\cd^{\rm b}(\mod A)$ has no presilting object $X$ such that $\add X$ contains $\proj A$ as a proper subcategory.
\end{conjecture}

\begin{conjecture}\label{Tachikawa 2}
There does not exist a thick subcategory $\ct$ of $\cd^{\rm b}(\mod A)$
containing $\ck^{\rm b}(\proj A)$ such that the Grothendieck group
$K_0(\ct)$ is a free abelian group with rank strictly bigger than $n$.
\end{conjecture}

We have the following observation (see also Section 4 of \cite{Happel93}).

\begin{theorem}
Conjecture \ref{Tachikawa 2} $\Rightarrow$ Conjecture \ref{Tachikawa 1.5} $\Rightarrow$ the Auslander--Reiten Conjecture.
\end{theorem}

\begin{proof}

To prove the first implication, assume that a non-projective $A$-module $X$ satisfies $\Ext^i_A(X,X\oplus A)=0$
for all $i>0$. Then $\ct:=\thick(X\oplus A)$ is a thick subcategory of
$\cd^{\rm b}(\mod A)$ containing $\ck^{\rm b}(\proj A)$, and $X\oplus A$ is a silting object in $\ct$. It is shown in \cite[Theorem 2.27]{AI}
that the Grothendieck group $K_0(\ct)$ is a free abelian group and
the rank is equal to the number of non-isomorphic indecomposable direct
summands of $X\oplus A$. Thus the assertion follows.

To obtain the second implication it suffices to observe that  if $X\in\mod A$ is not projective and satisfies $\Ext^i_A(X,X\oplus A)=0$ for all $i>0$, then $X\oplus A$ is a presilting object of $\cd^{\bo}(\mod A)$ such that $\add(X\oplus A)$ contains $\proj A$ as a proper subcategory.
\end{proof}

When $A$ is self-injective, the Auslander--Reiten Conjecture takes the following form due to Tachikawa. 

\medskip\noindent
{\bf The Tachikawa Conjecture} (\cite{Tachikawa73})
\emph{Assume that $A$ is self-injective. If $X\in\mod A$ satisfies $\Ext^i_A(X,X)=0$ for all $i>0$, then $X$ is a projective module.}

\medskip
Formulated in terms of presilting objects, it has the form:

\begin{conjecture}\label{Tachikawa3.5}
Assume that $A$ is self-injective. Then the stable category $\underline{\mod}A$ has no non-trivial presilting objects.
\end{conjecture}

By Theorems~\ref{t:silting-reduction} and~\ref{buchweitz}, this is equivalent to Conjecture~\ref{Tachikawa 1.5} for self-injective algebras. What we know is the following.

\begin{proposition} \emph{(\cite[Example 2.5]{AI})}
Assume that $A$ is self-injective. Then the stable category $\underline{\mod}A$ has no silting objects unless $A$ is semi-simple.
\end{proposition}



\def\cprime{$'$}
\providecommand{\bysame}{\leavevmode\hbox to3em{\hrulefill}\thinspace}
\providecommand{\MR}{\relax\ifhmode\unskip\space\fi MR }
\providecommand{\MRhref}[2]{%
  \href{http://www.ams.org/mathscinet-getitem?mr=#1}{#2}
}
\providecommand{\href}[2]{#2}

\end{document}